\documentclass[10pt]{article}
\usepackage{latexsym, amsmath, amssymb, amsthm}
\usepackage{authblk, setspace, enumerate}
\usepackage{graphicx, tikz, cite}
\usepackage[top=1 in, bottom=1 in, left=1 in, right=1 in]{geometry}

\newtheorem{theorem}{Theorem}[section]
\newtheorem{lemma}[theorem]{Lemma}
\newtheorem{corollary}[theorem]{Corollary}
\newtheorem{proposition}[theorem]{Proposition}
\newtheorem{remark}[theorem]{Remark}
\newtheorem{definition}[theorem]{Definition}

\numberwithin{equation}{section}
 
\def\p{\partial}  \def\ora{\overrightarrow}
		\def\ol{\overline}		\def\m{\mathbb}		
\def\O{\Omega}  \def\lam{\lambda}  \def\eps{\epsilon}  
\def\t{\tilde}	\def\wt{\widetilde}
\def\be{\begin{equation}}     \def\ee{\end{equation}}
\def\bea{\begin{eqnarray}}     \def\eea{\end{eqnarray}}
\def\beas{\begin{eqnarray*}}     \def\eeas{\end{eqnarray*}}

\title{Lifespan estimates via Neumann heat kernel}
\author[a]{Xin Yang\thanks{Email: yang2x2@ucmail.uc.edu}}
\author[b]{Zhengfang Zhou\thanks{Email: zfzhou@math.msu.edu}}
\affil[a]{Department of Mathematical Sciences, University of Cincinati, Cincinnati, OH 45221, USA}
\affil[b]{Department of Mathematics, Michigan State University, East Lansing, MI 48824, USA}

\date{}

\begin{document}
\maketitle

\begin{abstract}
This paper studies the lower bound of the lifespan $T^{*}$ for the heat equation $u_t=\Delta u$ in a bounded domain $\Omega\subset\mathbb{R}^{n}(n\geq 2)$ with positive initial data $u_{0}$ and a nonlinear radiation condition on partial boundary: the normal derivative $\partial u/\partial n=u^{q}$ on $\Gamma_1\subseteq \partial\Omega$ for some $q>1$, while $\partial u/\partial n=0$ on the other part of the boundary. Previously, under the convexity assumption of $\Omega$, the asymptotic behaviors of $T^{*}$ on the maximum $M_{0}$ of $u_{0}$ and the surface area $|\Gamma_{1}|$ of $\Gamma_{1}$ were explored. In this paper, without the convexity requirement of $\Omega$, we will show that as $M_{0}\rightarrow 0^{+}$, $T^{*}$ is  at least of order $M_{0}^{-(q-1)}$ which is optimal. Meanwhile, we will also prove that as $|\Gamma_{1}|\rightarrow 0^{+}$, $T^{*}$ is at least of order $|\Gamma_{1}|^{-\frac{1}{n-1}}$ for $n\geq 3$ and $|\Gamma_{1}|^{-1}\big/\ln\big(|\Gamma_{1}|^{-1}\big)$ for $n=2$. The order on $|\Gamma_{1}|$ when $n=2$ is almost optimal. The proofs are carried out by analyzing the representation formula of $u$ in terms of the Neumann heat kernel. 
\end{abstract}


\bigskip
\bigskip

\section{Introduction}  
\label{Sec, introduction}
\subsection{Problem and Results}
In this paper, $\Omega$ represents a bounded open subset in $\mathbb{R}^{n}$ ($n\geq 2$) with $C^{2}$ boundary $\partial\Omega$. $\Gamma_1$ and $\Gamma_2$ denote two disjoint relatively open subsets of $\p\O$ such that $\Gamma_{1}\neq \emptyset$ and $\ol{\Gamma}_1\cup\ol{\Gamma}_2=\p\O$. Moreover, the interface $\wt{\Gamma}$, defined by $\widetilde{\Gamma}=\ol{\Gamma}_1\cap\ol{\Gamma}_2$, is the common boundary of $\Gamma_{1}$ and $\Gamma_{2}$. We assume $\wt{\Gamma}$
is $C^{1}$ as the boundary of $\Gamma_1$ or $\Gamma_2$. 
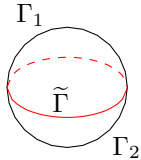
\begin{figure}[hbtp]
\centering
\begin{tikzpicture}[scale=0.8]
\draw [domain=0:360] plot ({cos(\x)}, {sin(\x)});
\draw [color=red] [dashed] [domain=-1:1] plot ({\x},{((1-(\x)^2)/4)^(1/2)});
\draw [color=red] [domain=-1:1] plot ({\x},{-((1-(\x)^2)/4)^(1/2)});
\draw (-1,1.2) node [right] {$\Gamma_{1}$};
\draw (0.6,-1) node [right] {$\Gamma_{2}$};
\draw (-0.4,-0.2) node [right] {$\wt{\Gamma}$};
\end{tikzpicture}
\caption{Interface $\widetilde{\Gamma}$}
\label{Fig, interface}
\end{figure}
For example (see Figure \ref{Fig, interface}), if $\O$ is a ball, and $\Gamma_{1}$ and $\Gamma_{2}$ are the open upper and lower hemispheres, then the interface $\wt{\Gamma}$ is the equator.

We study the following problem:
\be\label{Prob}
\left\{\begin{array}{lll}
(\p_{t}-\Delta_{x})u(x,t)=0 &\text{in}& \Omega\times (0,T], \vspace{0.02in}\\
\dfrac{\partial u(x,t)}{\partial n(x)}=u^{q}(x,t) &\text{on}& \Gamma_1\times (0,T],\vspace{0.04in}\\
\dfrac{\partial u(x,t)}{\partial n(x)}=0 &\text{on}& \Gamma_2\times (0,T], \vspace{0.02in}\\
u(x,0)=u_0(x) &\text{in}& \Omega,
\end{array}\right.\ee
where 
\be\label{assumption on prob}
q>1,\, u_0\in C^{1}(\ol{\O}),\, u_0(x)\geq 0,\, u_0(x)\not\equiv 0. \ee
The normal derivative on the boundary is understood in the classical way: for any $(x,t)\in\p\O\times(0,T]$,
\be\label{Normal deri def, classical}
\frac{\p u(x,t)}{\p n(x)}\triangleq \lim_{h\rightarrow 0^{+}} \frac{u(x,t)-u(x-h\ora{n}(x),t)}{h},\ee
where $\ora{n}(x)$ denotes the exterior unit normal vector at $x$. $\p\O$ being $C^2$ ensures that $x-h\ora{n}(x)$ belongs to $\O$ when $h$ is positive and sufficiently small. 

Throughout this paper, we write
\be\label{initial max}
M_0= \max_{x\in\ol{\O}}u_0(x)\ee
and denote $M(t)$ to be the supremum of the solution $u$ to (\ref{Prob}) on $\ol{\O}\times[0,t]$:
\be\label{max function at time t}
M(t)=\sup_{(x,\tau)\in \ol{\O}\times[0,t]}u(x,\tau).\ee
$|\Gamma_{1}|$ represents the surface area of $\Gamma_{1}$, that is
\[|\Gamma_{1}|=\int_{\Gamma_{1}}\,dS,\]
where $dS$ means the surface integral. $\Phi$ refers to the heat kernel of $\m{R}^{n}$:
\be\label{fund soln of heat eq}
\Phi(x,t)=\frac{1}{(4\pi t)^{n/2}}\,\exp\Big(-\frac{|x|^2}{4t}\Big), \quad\forall\, (x,t)\in\m{R}^{n}\times(0,\infty).\ee
In addition, $C=C(a,b\dots)$ and $C_{i}=C_{i}(a,b\dots)$ stand for positive and finite constants which only depend on the parameters $a,b\dots$. One should also note that $C$ and $C_{i}$ may represent different constants in different places.

The recent paper \cite{YZ16} studied (\ref{Prob}) systematically and the motivation was the disaster of the Space Shuttle Columbia (see Figure \ref{Fig, Columbia}) in 2003, we refer the reader to that paper for the detailed discussion of the background.
\begin{figure}[!ht]
\centering
\begin{tikzpicture}[scale=0.6]
\begin{large}
\draw (-5,1/2)-- (0,1/2);
\draw (0,1/2)--(5/2,7/2);
\draw [domain=5/2:3] plot ({\x},{7/2+1/16-(\x-11/4)^2});
\draw (3,7/2)--(3,1/2);
\draw (3,1/2)--(4,1/2);
\draw (-5,-1/2)-- (0,-1/2);
\draw (0,-1/2)--(5/2,-7/2);
\draw [domain=5/2:3] plot ({\x},{-7/2-1/16+(\x-11/4)^2});
\draw (3,-7/2)--(3,-1/2);
\draw (3,-1/2)--(4,-1/2);
\draw [domain=90:270] plot ({cos(\x)-5},{1/2*sin(\x)});
\draw [dashed] [domain=0:360] plot({-5+1/8*cos(\x)},{1/2*sin(\x)});
\draw [dashed] [domain=0:360] plot({4+1/8*cos(\x)},{1/2*sin(\x)});
\path (-1.5,2.2) coordinate (A);
\draw (A) node [above] {$u$: temperature};
\draw [color=blue] [domain=110:150] plot({5/4+3/4*cos(\x)},{-2+3/4*sin(\x)});
\draw [color=blue] [domain=290:330] plot({5/4+3/4*cos(\x)},{-2+3/4*sin(\x)});
\path (5/4,-2+1/4) coordinate (B);
\draw (B) node [right] {$\Gamma_1$};
\path (5/4+1/4,2-1/4) coordinate (D);
\draw (D) node [below] {$\Omega$};
\path (-2,1/2) coordinate (F);
\draw (F) node [above] {$\Gamma_2$};
\path (1/2,-5/2) coordinate (E);
\end{large}
\end{tikzpicture}
\caption{Space Shuttle Columbia}
\label{Fig, Columbia}
\end{figure}
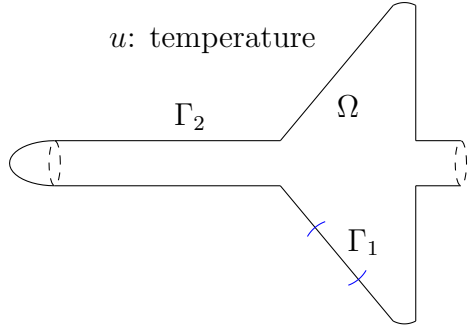
As a summary of its conclusions, \cite{YZ16} first established the local existence and uniqueness theory for (\ref{Prob}) in the following sense: there exist $T>0$ and a unique solution $u$ in $C^{2,1}\big(\O\times(0,T]\big)\bigcap C\big(\overline{\O}\times[0,T]\big)$ which satisfies (\ref{Prob}) pointwisely and also satisfies 
\be\label{interface bdry deri}
\frac{\p u(x,t)}{\p n(x)}=\frac{1}{2}\,u^{q}(x,t), \quad \forall\, (x,t)\in\wt{\Gamma}\times (0,T]. \ee
Moreover, it is shown that this unique solution $u$ is strictly positive when $t>0$. We want to remark here that the solution constructed in \cite{YZ16} through the heat potential technique automatically satisfies (\ref{interface bdry deri}) due to a generalized jump relation (See Theorem A.3 in \cite{YZ16}). The purpose of imposing this additional restriction (\ref{interface bdry deri}) to the local solution is to ensure the uniqueness through the Hopf's lemma, it is not clear whether the uniqueness will still hold without this restriction. After the local existence and uniqueness theory was set up, \cite{YZ16} also studied the blow-up phenomenon of (\ref{Prob}). If $T^{*}$ denotes the lifespan (maximal existence time) of the local solution $u$, then it is proved that $T^{*}<\infty$ and 
\be\label{supremum norm blows up}
\lim\limits_{t\nearrow T^{*}}M(t)=\infty.\ee
So the lifespan $T^{*}$ is exactly the blow-up time of $u$. Moreover, if $\min\limits_{\ol{\O}}u_{0}>0$, then $T^{*}$ has the following upper bound:
\be\label{upper bdd}
T^{*}\leq \frac{1}{(q-1)|\Gamma_1|}\int_{\O}u_0^{1-q}(x)\,dx.\ee
Meanwhile, \cite{YZ16} also provides a lower bound for $T^{*}$.

Later in \cite{YZ18}, it improves the lower bound as below.
\be\label{lower bdd, old}
T^{*}\geq \frac{C}{q-1}\ln\Big(1+(2M_{0})^{-4(q-1)}\,|\Gamma_{1}|^{-\frac{2}{n-1}}\Big),\ee
where $C=C(n,\O)$. Based on (\ref{upper bdd}) and (\ref{lower bdd, old}), if $q\rightarrow 1^{+}$, then both the upper and lower bounds of $T^{*}$ tends to infinity at the order $(q-1)^{-1}$, which implies the order of $T^{*}$ is exactly $(q-1)^{-1}$. On the other hand, if $|\Gamma_{1}|\rightarrow 0^{+}$, then the order of the upper bound is $|\Gamma_{1}|^{-1}$ while the lower bound only has a logarithmic order of $|\Gamma_{1}|^{-1}$. Similarly, if the initial maximum $M_{0}\rightarrow 0^{+}$, then the order of the upper bound (by assuming $u_{0}$ is comparable to $M_{0}$) is $M_{0}^{-(q-1)}$ while the lower bound only has a logarithmic order of $M_{0}^{-1}$. So it is natural to ask that as $|\Gamma_{1}|\rightarrow 0^{+}$ (resp. $M_{0}\rightarrow 0^{+}$), whether the lower bound can be improved to be of order $|\Gamma_{1}|^{-\alpha}$ (resp. $M_{0}^{-\alpha}$) for some $\alpha>0$? In \cite{YZ18}, it gives an affirmative answer to this question {\bf under the convexity assumption of the domain $\O$}. However, in many situations, the domain $\O$ may not be convex. For example, as the motivation of the problem (\ref{Prob}) illustrated in \cite{YZ16}, the Space Shuttle Columbia (see Figure \ref{Fig, Columbia}) is not convex. So the main goal of this paper is to remove the convexity assumption, and we will apply a new approach which takes advantage of the Neumann heat kernel.

Previously, the methods used in \cite{YZ16} and \cite{YZ18} based on the representation formula (see Corollary 3.9 in \cite{YZ16}) of $u$ in terms of the heat kernel $\Phi$ of $\m{R}^{n}$. More precisely, for any boundary point $x\in\p\O$ and $t\in[0,T^{*})$,
\be\begin{split}
u(x,t) = &\,\, 2\int_{\O}\Phi(x-y,t)\,u_{0}(y)\,dy-2\int_{0}^{t}\int_{\p\O}\frac{\p\Phi(x-y,t-\tau)}{\p n(y)}\,u(y,\tau)\,dS(y)\,d\tau \\
&+2\int_{0}^{t}\int_{\Gamma_1}\Phi(x-y,t-\tau)\,u^{q}(y,\tau)\,dS(y)\,d\tau, \label{rep formula by heat kernel}
\end{split}\ee
where 
\[\frac{\p\Phi(x-y,t-\tau)}{\p n(y)}=-(D\Phi)(x-y,t-\tau)\cdot \ora{n}(y).\]
We want to remark that there also exists a representation formula for the inside point $x\in\O$ and $t\in[0,T^{*})$, see Theorem 3.8 in \cite{YZ16}. That is, 
\be\begin{split}
u(x,t) = &\,\, \int_{\O}\Phi(x-y,t)\,u_{0}(y)\,dy-\int_{0}^{t}\int_{\p\O}\frac{\p\Phi(x-y,t-\tau)}{\p n(y)}\,u(y,\tau)\,dS(y)\,d\tau \\
&+\int_{0}^{t}\int_{\Gamma_1}\Phi(x-y,t-\tau)\,u^{q}(y,\tau)\,dS(y)\,d\tau. \label{rep formula by heat kernel, inside}
\end{split}\ee
The formula (\ref{rep formula by heat kernel, inside}) is different from (\ref{rep formula by heat kernel}) in that the coefficients 2's do not appear in front of the integrals on the right hand side. The existence of the coefficients 2's in (\ref{rep formula by heat kernel}) is due to the jump relation of the single-layer heat potential when $x\in\p\O$ (see e.g. Corollary Appendix A.2 in \cite{YZ16} or Theorem 9.5, Sec. 2, Chap. 9 in \cite{Kre14}). The drawback of the formula (\ref{rep formula by heat kernel}) is the uncertainty of the sign of the term $\frac{\p\Phi(x-y,t-\tau)}{\p n(y)}$. It is this reason that demands the convexity of $\O$ in \cite{YZ18} to improve the lower bound of $T^{*}$ to a power order of $|\Gamma_{1}|^{-1}$ (resp. $M_{0}^{-1}$) as $|\Gamma_{1}|\rightarrow 0^{+}$ (resp. $M_{0}\rightarrow 0^{+}$). 

In order to avoid the integral containing $\frac{\p\Phi(x-y,t-\tau)}{\p n(y)}$, it motivates us to consider the Green's function $G(x,t,y,s)$ of the heat operator in $\O$ with the Neumann boundary condition whose normal derivative vanishes (see (\ref{Green fn solves heat and bdry}) in Lemma \ref{Lemma, prop of Green fn}). As a convention, $G(x,t,y,s)$ is also called the Neumann Green's function.  In addition, since the coefficients of the heat operator are constants, the Neumann Green's function is invariant under time translation (see part(f) in Lemma \ref{Lemma, prop of Green fn}). So it is more convenient to consider the corresponding Neumann heat kernel $N(x,y,t)$ of $\O$ (see Definition \ref{Def, NHK}). Unlike the heat kernel $\Phi$ of $\m{R}^{n}$, the Neumann heat kernel $N(x,y,t)$ of $\O$ does not have an explicit formula in general which makes it difficult to quantify. Fortunately, for small time $t$, $N(x,y,t)$ can be bounded in terms of $\Phi$ (see Lemma \ref{Lemma, quant of NHK}). This property will help us to justify the representation formula (\ref{rep for soln, initial}) via $N(x,y,t)$ and further analyze it. In fact, if the solution $u$ to (\ref{Prob}) is smooth, then it is straightforward to obtain (\ref{rep for soln, initial}) based on the properties of $N(x,y,t)$ in Corollary \ref{Cor, prop of NHK}. Now although $u$ is not smooth near the boundary $\p\O$ or near the initial time $t=0$, by taking advantage of Lemma \ref{Lemma, quant of NHK}, we are able to verify (\ref{rep for soln, initial}) in a way similar to the proof for (\ref{rep formula by heat kernel, inside}) in \cite{YZ16}. In contrast to (\ref{rep formula by heat kernel, inside}), the formula (\ref{rep for soln, initial}) does not contain any term that may cause the jump relation along the boundary $\p\O$. As a result, the formula (\ref{rep for soln, initial}) holds for both $x\in\O$ and $x\in\p\O$. In addition, due to Lemma \ref{Lemma, quant of NHK} again, for small time $t$, the estimate on the term $\int_{0}^{t}\int_{\Gamma_{1}}N(x,y,t-\tau)u^{q}(y,\tau)\,dS(y)\,d\tau$ boils down to the estimate on $\int_{0}^{t}\int_{\Gamma_{1}}\Phi\big(x-y,2(t-\tau)\big)u^{q}(y,\tau)\,dS(y)\,d\tau$ which has been treated in \cite{YZ18}. Finally, noticing that the method in \cite{YZ18} analyzes the representation formula discretely and in each step the time is indeed small, so we can combine that method with (\ref{rep for soln, initial}) and Lemma \ref{Lemma, quant of NHK} to achieve our goal. The following are the main results of this paper.

\begin{theorem}\label{Thm, lower bdd, general case}
Assume (\ref{assumption on prob}) and let $T^{*}$ be the lifespan for (\ref{Prob}). Then there exists a constant $C=C(n,\O)$ such that 
\be\label{lower bdd, general case}
T^{*}\geq \frac{C}{q-1}\ln\Big(1+(2M_{0})^{-4(q-1)}\,|\Gamma_{1}|^{-\frac{2}{n-1}}\Big).\ee
\end{theorem}

This theorem is not new and it has appeared in Theorem 1.1 of \cite{YZ18}. But its proof in this paper, as mentioned above, is different and based on the representation formula (\ref{rep for soln, initial}) involving the Neumann heat kernel $N(x,y,t)$. 

\begin{theorem}\label{Thm, lower bdd}
Assume (\ref{assumption on prob}) and let $T^{*}$ be the lifespan for (\ref{Prob}). Denote $M_{0}$ as in (\ref{initial max}) and define 
\[ Y=\left\{\begin{array}{lll}
M_{0}^{q-1}|\Gamma_{1}|^{\frac{1}{n-1}}, & \text{if}\quad n\geq 3, \vspace{0.1in}\\
M_{0}^{q-1}|\Gamma_{1}|\ln\Big(1+\dfrac{1}{|\Gamma_{1}|}\Big), & \text{if}\quad n=2.
\end{array}\right.\]
Then there exist constants $Y_{0}=Y_{0}(n,\O)$ and $C=C(n,\O)$ such that if $Y\leq \frac{Y_{0}}{q}$, then
\be\label{lower bdd}
T^{*}\geq \frac{C}{(q-1)Y}.\ee
\end{theorem}

The lower bound (\ref{lower bdd}) was also obtained in \cite{YZ18} {\bf under the convexity assumption of $\O$}, so the significance of Theorem \ref{Thm, lower bdd} is the {\bf removal of the convexity requirement}. The method in the proof again relies on the representation formula (\ref{rep for soln, initial}). 

\begin{remark}\label{Remark, analysis of asym behavior}
From Theorem \ref{Thm, lower bdd}, we draw two conclusions.
\begin{enumerate}[(1)]
\item Relation between $T^{*}$ and $M_{0}$: if $M_{0}\rightarrow 0^{+}$ and other factors are fixed, then (\ref{lower bdd}) implies 
$$T^{*}\geq C_{1}\,M_{0}^{-(q-1)}.$$
This order is optimal since if the initial data $u_{0}$ is a constant function, then it follows from (\ref{upper bdd}) that 
\[T^{*}\leq C_{2}\,M_{0}^{-(q-1)}.\]

\item  Relation between $T^{*}$ and $|\Gamma_{1}|$: if $|\Gamma_{1}|\rightarrow 0^{+}$ and other factors are fixed,
then it follows from (\ref{upper bdd}) that $T^{*}$ is at most of order $|\Gamma_{1}|^{-1}$. On the other hand, (\ref{lower bdd}) implies that $T^{*}$ is at least of order $|\Gamma_{1}|^{-\frac{1}{n-1}}$ for $n\geq 3$ and $|\Gamma_{1}|^{-1}\big/\ln\big(|\Gamma_{1}|^{-1}\big)$ for $n=2$. In particular when $n=2$, the order of the lower bound is almost optimal (within a logarithmic order to the upper bound).
\end{enumerate}
\end{remark}

\subsection{Historical Works}
\subsubsection{Blow-up phenomenon for the heat equation with nonlinear Neumann conditions}
Since the pioneering papers by Kaplan \cite{Kap63} and Fujita\cite{Fuj66}, the blow-up phenomenon of parabolic type has been extensive studied in the literature for the Cauchy problem as well as the boundary value problems. We refer the readers to the surveys \cite{DL00, Lev90}, the books \cite{Hu11, QS07} and the references therein. 

One of the typical problems is the heat equation with Neumann boundary conditions in a bounded domain $\O$:
\be\label{heat with Neumann}
\left\{\begin{array}{lll}
(\p_{t}-\Delta_{x})u(x,t)=0 &\text{in}& \Omega\times (0,T], \vspace{0.02in}\\
\dfrac{\partial u(x,t)}{\partial n(x)}=F\big(u(x,t)\big) &\text{on}& \p\O\times (0,T], \vspace{0.02in}\\
u(x,0)=\psi(x) &\text{in}& \Omega.
\end{array}\right.\ee
Here, the initial data $\psi$ is not assumed to be nonnegative. It is well-known that there are two ways to construct the classical solution to (\ref{heat with Neumann}) depending on the smoothness of $\p\O$, $F$ and $u_{0}$ (see Theorem 1.1 and 1.3 in \cite{L-GMW91}, also see the books \cite{Fri64, Lie96, LSU68}).
\begin{itemize}
\item[(a)] The first way is by Schauder estimate. Assume $\p\O$ is $C^{2+\alpha}$, $F\in C^{1+\alpha}(\m{R})$, $\psi\in C^{2+\alpha}(\ol{\O})$ and the compatibility condition
\[\frac{\p \psi(x)}{\p n(x)}=F\big(\psi(x)\big), \quad \forall\, x\in \p\O.\]

Then there exist $T>0$ and a unique function $u$ in $C^{2+\alpha,1+\frac{\alpha}{2}}\big(\ol{\O}\times[0,T]\big)$ which satisfies (\ref{heat with Neumann}) pointwisely.

\item[(b)] The second way is by the heat potential technique. The requirements on the data can be relaxed and in particular, the compatibility condition is no longer needed, but accordingly the conclusion is also weaker. More precisely, assume $\p\O$ is $C^{1+\alpha}$, $F\in C^{1}(\m{R})$ and $\psi\in C^{1}(\ol{\O})$. Then there exist $T>0$ and a unique function $u$ in $C^{2,1}\big(\O\times(0,T]\big)\bigcap C\big(\overline{\O}\times[0,T]\big)$ which satisfies (\ref{heat with Neumann}) pointwisely.
\end{itemize}
In most papers, the assumptions will fall into or similar to either case (a) or case (b). In the following statements, we will ignore their distinctions and just refer them to be the local (classical) solutions. 

It has been already known that if $F$ is bounded on $\m{R}$, then the local solutions can be extended globally. But if $F$ is unbounded, then the finite-time blowup may occur. The first result on the blow-up phenomenon for (\ref{heat with Neumann}) is due to Levine and Payne \cite{LP74}. They used a concavity argument to conclude that any classical solution blows up in finite time under the two assumptions below.
\begin{itemize}
\item First, 
\be\label{LP radiation fn}
F(z)=|z|^{q}h(z),\ee
for any constant $q>1$ and any differentiable, non-decreasing function $h(z)$.

\item Secondly,  
\be\label{LP cond}
\frac{1}{|\p\O|}\,\int_{\p\O}\bigg(\int_{0}^{\psi(x)}F(z)\,dz\bigg)\,dS(x)>\frac{1}{2}\int_{\O}|D\psi(x)|^{2}\,dx. \ee
\end{itemize}
\begin{remark}\label{Remark, cor of LP74} As a corollary of the result in \cite{LP74}, if $h(z)$ in (\ref{LP radiation fn}) is also positive and $\psi$ is a positive constant function, then (\ref{LP cond}) is satisfied and therefore the solution blows up in finite time. Combining this fact with the comparison principle, it implies that for any positive $h(z)$ in (\ref{LP radiation fn}) and for any positive $\psi$, the solution blows up in finite time. \end{remark} 
Later, Walter \cite{Wal75} gave a more complete characterization for the blow-up phenomenon by introducing some comparison functions. More precisely, if $F(z)$ is positive, increasing and convex for $z\geq z_{0}$ with some constant $z_{0}$, then there are exactly two possibilities.
\begin{itemize}
\item First, if $\int_{z_{0}}^{\infty}\frac{1}{F(z)F'(z)}\,dz=\infty$, then the solution exists globally for any initial data $\psi$.

\item Secondly, if $\int_{z_{0}}^{\infty}\frac{1}{F(z)F'(z)}\,dz<\infty$, then the solution blows up in finite time for large initial data $\psi$.
\end{itemize}
The result was further generalized by Rial and Rossi \cite{RR97} (also see \cite{L-GMW91}). In \cite{RR97}, by assuming $F$ to be $C^{2}$, increasing and positive in $\m{R}_{+}$, and also assuming $1/F$ to be locally integrable near $\infty$ (that is $\int^{\infty}\frac{1}{F(z)}\,dz<\infty$), it is shown that for any positive initial data $\psi$, the classical solution blows up in finite time. The success of their method was due to a clever choice of an energy function which made the proof short and elementary.

Applying these earlier results to the simpler model (that is (\ref{Prob}) with $\Gamma_{2}=\emptyset$)
\be\label{heat with Neumann, power}
\left\{\begin{array}{lll}
(\p_{t}-\Delta_{x})u(x,t)=0 &\text{in}& \Omega\times (0,T], \vspace{0.02in}\\
\dfrac{\partial u(x,t)}{\partial n(x)}=u^{q}(x,t) &\text{on}& \p\O\times (0,T], \vspace{0.02in}\\
u(x,0)=u_{0}(x) &\text{in}& \Omega,
\end{array}\right.\ee
where $q>1$ and the initial data $u_{0}\geq 0$ and $u_{0}\not\equiv 0$, it can be shown that any solution to (\ref{heat with Neumann, power}) blows up in finite time. In fact, by the maximum principle, the solution $u$ becomes positive as soon as $t>0$. Then either Remark \ref{Remark, cor of LP74} or the result in \cite{RR97} (also see \cite{HY94}) implies the finite time blowup of the solution. However, when the nonlinear radiation condition is only imposed on partial boundary (that is when $\Gamma_{2}\neq\emptyset$ in (\ref{Prob})), additional difficulties appear due to the discontinuity of the normal derivative along the interface $\wt{\Gamma}$ between $\Gamma_{1}$ and $\Gamma_{2}$. To our knowledge, \cite{YZ16} and \cite{YZ18} were the first papers that dealt with this problem and quantified both upper and lower bounds of the lifespan (or equivalently the blow-up time).

\subsubsection{Lower bound estimate for the lifespan}
When considering the bounds of the lifespan, the upper bound is usually related to the nonexistence of the global solutions and various methods on this issue have been developed (see \cite{Lev75} for a list of six methods). The lower bound was not studied as much in the past and not many methods have been explored. However, the lower bound may be more useful in practice since it serves as the safe time. In the existing literature, the most common ideas are the comparison argument and the differential inequality techniques. 

The first work on the lower bound estimate of the lifespan was due to Kaplan \cite{Kap63}. Later, Payne and Schaefer  developed a very robust method on this issue. For example, they derived the lower bound of the lifespan for the nonlinear heat equation with homogeneous Dirichlet or Neumann boundary conditions in \cite{PS06, PS07}. Later this idea was also applied to the problem (\ref{heat with Neumann}) (see \cite{PS09}) and many other types of problems (see e.g. \cite{Ena11, PPV-P10, PP13, BS14, LL12, TV-P16, AD17, DS16}). However, this method requires the domain to be convex. In addition, it is not effective to deal with the partial nonlinear boundary conditions like the one in (\ref{Prob}).

Recently, in order to obtain the lower bound of the lifespan for the problem (\ref{Prob}), the authors of this paper developed a new method in \cite{YZ18} by discretely analyzing the representation formula of the solution in terms of $\Phi$. Firstly, without the convexity assumption, \cite{YZ18} obtained a lower bound for $T^{*}$ which was logarithmic order of $|\Gamma_{1}|^{-1}$ (resp. $M_{0}^{-1}$) as $|\Gamma_{1}|\rightarrow 0^{+}$ (resp. $M_{0}\rightarrow 0^{+}$). On the other hand, {\bf by assuming $\O$ to be convex}, it improved the lower bound to be of power order of $|\Gamma_{1}|^{-1}$ (resp. $M_{0}^{-1}$) when $|\Gamma_{1}|\rightarrow 0^{+}$ (resp. $M_{0}\rightarrow 0^{+}$) as in Theorem \ref{Thm, lower bdd}. 

In the current paper, {\bf without the convexity assumption on $\O$}, by combining the method in \cite{YZ18} with the new representation formula in terms of the Neumann heat kernel $N(x,y,t)$ (see (\ref{rep for soln, initial})), we are able to show that $T^{*}$ is at least the same power order of $|\Gamma_{1}|^{-1}$ (resp. $M_{0}^{-1}$) when $|\Gamma_{1}|\rightarrow 0^{+}$ (resp. $M_{0}\rightarrow 0^{+}$) as in \cite{YZ18}.

\subsection{Organization}
The organization of this paper is as follows. Section \ref{Sec, pre} will introduce the definitions and the basic properties of the Neumann Green's function and the Neumann heat kernel of $\O$. In addition, it will discuss the representation formula of the solution and provide two crucial estimates on the boundary-time integrals of the heat kernel $\Phi$ of $\m{R}^{n}$. Section \ref{Sec, proof for general lower bdd} and Section \ref{Sec, proof for main thm} will prove Theorem \ref{Thm, lower bdd, general case} and Theorem \ref{Thm, lower bdd} respectively. Section \ref{Sec, sharp} will demonstrate the sharpness of Lemma \ref{Lemma, bdry-time int, critical} which plays an essential role in Section \ref{Sec, proof for main thm}. Finally in the Appendix \ref{Sec, proof for rep formula}, a rigorous proof will be given to the representation formula mentioned in Section \ref{Sec, pre} which is the key tool in this paper.

\section{Preliminaries}\label{Sec, pre}
\subsection{Neumann Green's Function and Neumann Heat Kernel}
\label{Subsec, NGF and NHK}
Given a bounded domain $\O$ in $\m{R}^{n}$ and a parabolic operator $L$ on $\O$, similar to the elliptic case, one can define the fundamental solution associated to $L$ on $\O$ (see e.g. \cite{Dre40, Fel37, Ito53}). If in addition the boundary conditions are considered, one can also study the fundamental solution adapted to the boundary conditions (see e.g. \cite{Ito54, Ito57a, Ito57b}). Such a fundamental solution with the boundary condition is usually called the Green's function. In particular, if the boundary condition is of Neumann type, then the associated fundamental solution is called the Neumann Green's function. When the coefficients of the parabolic operator $L$ are independent of the time $t$, the Neumann Green's function is invariant under the time translation. Consequently, it automatically generates a Neumann heat kernel which has a simpler form but captures the essential properties of the Neumann Green's function. The operator considered in this paper is just the heat operator $L=\p_{t}-\Delta_{x}$ whose coefficients are constants, so we will first state the precise definitions of the associated Neumann Green's function and the Neumann heat kernel, and then collect some classical properties which are needed later.

Roughly speaking, for the heat operator 
\be\label{heat op}
L_{tx}=\p_{t}-\Delta_{x}\ee
with the Neumann boundary condition, the associated Green's function on $\O$, which is also called the Neumann Green's function on $\O$, is  a function $G(x,t,y,s)$ defined on $\{(x,t,y,s): x,y\in\ol{\O}, t,s\in\m{R}, s<t\}$ such that for any fixed $s\in\m{R}$ and $y\in\ol{\O}$,
\[\left\{\begin{array}{rll}
(\p_{t}-\Delta_{x})G(x,t,y,s) &=0, & \quad\forall\,x\in\ol{\O},\,t>s, \vspace{0.02in}\\
\dfrac{\p G(x,t,y,s)}{\p n(x)} &=0, & \quad\forall\,x\in\p\O,\, t>s, \vspace{0.02in}\\
\lim\limits_{t\rightarrow s^{+}} G(x,t,y,s) &=\delta(x-y), & \quad\text{in distributional sense.}
\end{array}\right.\]
In other words, for any fixed $s\in\m{R}$ and for any test function $\psi$ that satisfies 
\be\label{compatible initial}
\psi\in C(\ol{\O})\quad \text{and}\quad \frac{\p\psi(x)}{\p n(x)}=0,\quad\forall\, x\in\p\O,\ee
the function $v(x,t)$ defined as 
\[v(x,t)=\int_{\O}G(x,t,y,s)\psi(y)\,dy, \quad\forall\,x\in\ol{\O},\, t>s,\]
solves the following initial-boundary value problem:
\be\label{heat eq with Neumann cond}\left\{\begin{array}{lll}
(\p_{t}-\Delta_{x})v(x,t)=0 &\text{in}& \Omega\times (s,\infty), \vspace{0.02in}\\
\dfrac{\partial v(x,t)}{\partial n(x)}=0 &\text{on}& \p\O\times (s,\infty), \vspace{0.02in}\\
v(x,s)=\psi(x) &\text{in}& \Omega.
\end{array}\right.\ee

For the formal definition of the Neumann Green's function associated to the heat operator (\ref{heat op}) on $\O$, we follow (\cite{Ito54}, Page 171).
\begin{definition}\label{Def, NGF for heat}
Let $\O$ be a bounded domain in $\m{R}^{n}$ with $C^{2}$ boundary $\p\O$. Then we define the Neumann Green's function for the heat operator in $\O$ to be a continuous function $G(x,t,y,s)$ on $\{(x,t,y,s): x,y\in\ol{\O}, t,s\in\m{R}, s<t\}$ such that for any fixed $s\in\m{R}$ and for any $\psi$ in (\ref{compatible initial}), the function $v(x,t)$ defined as
\be\label{def of v}
v(x,t)=\int_{\O}G(x,t,y,s)\psi(y)\,dy\ee
belongs to $C^{2,1}\big(\ol{\O}\times (s,\infty)\big)$ and solves (\ref{heat eq with Neumann cond}) in the following sense:

\be\left\{\label{simple soln by Green fn}\begin{array}{rll}
(\p_{t}-\Delta_{x})v(x,t) &=0, & \quad\forall\,x\in\ol{\O},\,t>s, \vspace{0.02in}\\
\dfrac{\p v(x,t)}{\p n(x)} &=0, & \quad\forall\,x\in\p\O,\, t>s, \vspace{0.02in}\\
\lim\limits_{t\rightarrow s^{+}} v(x,t) &=\psi(x), & \quad\text{uniformly in $x\in\ol{\O}$}. 
\end{array}\right.\ee
\end{definition}

\begin{lemma}\label{Lemma, prop of Green fn}
Let $\O$ be a bounded domain in $\m{R}^{n}$ with $C^{2}$ boundary $\p\O$. Then there exists a unique Neumann Green's function $G(x,t,y,s)$ for the heat operator in $\O$ as in Definition \ref{Def, NGF for heat}. In addition, it has the following properties.
\begin{enumerate}[(a)]
\item $G(x,t,y,s)$ is $C^{2}$ in $x$ and $y$ ($x,y\in\ol{\O}$), and $C^{1}$ in $t$ and $s$ ($s<t$).

\item For fixed $s\in\m{R}$ and $y\in\ol{\O}$, as a function in $x$ and $t$ ($x\in\ol{\O}$ and $t>s$), $G(x,t,y,s)$ satisfies
\be\left\{\label{Green fn solves heat and bdry}\begin{array}{rl}
(\p_{t}-\Delta_{x})G(x,t,y,s) &=0, \quad\forall\,x\in\ol{\O},\, t>s, \vspace{0.02in}\\
\dfrac{\p G(x,t,y,s)}{\p n(x)} &=0, \quad\forall\,x\in\p\O,\, t>s.
\end{array}\right.\ee

\item For any $s\in\m{R}$ and $\psi$ in (\ref{compatible initial}), the function $v(x,t)$ defined in (\ref{def of v}) is the unique function in $C^{2,1}\big(\ol{\O}\times (s,\infty)\big)$ that satisfies (\ref{simple soln by Green fn}).

\item $G(x,t,y,s)\geq 0$ for any $x,y\in\ol{\O}$ and $s<t$.
\item $\int_{\O}G(x,t,y,s)\,dy=1$ for any $x\in\ol{\O}$ and $s<t$.
\item For any $x,y\in\ol{\O}$ and $s<t$, 
\[G(x,t,y,s)=G(x,t-s,y,0) \quad\text{and}\quad G(x,t,y,s)=G(y,t,x,s).\]
\end{enumerate}
\end{lemma}
\begin{proof}
The existence and uniqueness of the Neumann Green's function, and part (a)--(d) follow from Theorem 1--Theorem 4 in \cite{Ito54}. We just want to remark that although the regularity of $\p\O$ in \cite{Ito54} is required to be $C^{4,\gamma}$ for the general manifold and the general second order parabolic operators, here in the case of the Euclidean space and the heat operator, $\p\O$ being $C^{2}$ is enough. 

\begin{itemize}
\item For part (e), let $\psi\equiv 1$. Then $v\equiv 1$ obviously satisfies  (\ref{heat eq with Neumann cond}). Combining this fact with part (c) concludes (e). 

\item For part (f), since we are considering the heat equation whose coefficients are constants, it follows from Theorem 3 in \cite{Ito57b} that there exists a function $N(x,y,t)$ such that $G(x,t,y,s)=N(x,y,t-s)$. In addition, Theorem 4 in \cite{Ito57b} claims that $N(x,y,t)=N(y,x,t)$. Consequently, part (f) is justified.
\end{itemize} \end{proof}

As we have seen from the above proof that there exists a function $N(x,y,t)$ such that $N(x,y,t-s)=G(x,t,y,s)$. In particular, choosing $s=0$ leads to 
\be\label{def of NHK}
N(x,y,t)=G(x,t,y,0).\ee
This function $N(x,y,t)$ is called the Neumann heat kernel of $\O$.

\begin{definition}\label{Def, NHK}
Let $\O$ be a bounded domain in $\m{R}^{n}$ with $C^{2}$ boundary $\p\O$. A function $N(x,y,t)$ on $\ol{\O}\times\ol{\O}\times(0,\infty)$ is called a Neumann heat kernel if 
the function $G(x,t,y,s)$ defined by 
\[G(x,t,y,s)=N(x,y,t-s)\]
is the Neumann Green's function in Definition \ref{Def, NGF for heat}.
\end{definition}

Combining (\ref{def of NHK}) and Lemma \ref{Lemma, prop of Green fn}, we list some properties of the Neumann heat kernel.

\begin{corollary}\label{Cor, prop of NHK}
Let $\O$ be a bounded domain in $\m{R}^{n}$ with $C^{2}$ boundary $\p\O$. Then there exists a unique Neumann heat kernel $N(x,y,t)$ of $\O$ as in Definition \ref{Def, NHK}. In addition, it has the following properties. 
\begin{enumerate}[(a)]
\item $N(x,y,t)$ is $C^{2}$ in $x$ and $y$ ($x,y\in\ol{\O}$), and $C^{1}$ in $t$ ($t>0$).

\item For fixed $y\in\ol{\O}$, as a function in $(x,t)$, $N(x,y,t)$ satisfies
\be\left\{\label{NHK solves heat and bdry}\begin{array}{rl}
(\p_{t}-\Delta_{x})N(x,y,t) &=0, \quad\forall\,x\in\ol{\O},\, t>0, \vspace{0.02in}\\
\dfrac{\p N(x,y,t)}{\p n(x)} &=0, \quad\forall\,x\in\p\O,\, t>0.
\end{array}\right.\ee

\item For fixed $\psi$ in (\ref{compatible initial}), the function $w(x,t)$ defined by
\be\label{def of w}
w(x,t)=\int_{\O}N(x,y,t)\psi(y)\,dy\ee
is the unique function in $C^{2,1}\big(\ol{\O}\times (0,\infty)\big)$ that satisfies the following equations.
\be\left\{\label{simple soln by NHK}\begin{array}{rll}
(\p_{t}-\Delta_{x})w(x,t) &=0,  &\quad\forall\,x\in\ol{\O},\,t>0, \vspace{0.02in}\\
\dfrac{\p w(x,t)}{\p n(x)} &=0, &\quad\forall\,x\in\p\O,\, t>0, \vspace{0.02in}\\
\lim\limits_{t\rightarrow 0^{+}} w(x,t) &=\psi(x), &\quad\text{uniformly in $x\in\ol{\O}$}. 
\end{array}\right.\ee

\item $N(x,y,t)\geq 0$ and $N(x,y,t)=N(y,x,t)$  for any $x,y\in\ol{\O}$ and $t>0$.
\item $\int_{\O}N(x,y,t)\,dy=1$ for any $x\in\ol{\O}$ and $t>0$.
\end{enumerate}
\end{corollary}
\begin{proof}
These are direct consequences of Definition \ref{Def, NGF for heat}, Lemma \ref{Lemma, prop of Green fn} and Definition \ref{Def, NHK}.
\end{proof}

Unlike the heat kernel $\Phi$ of $\m{R}^{n}$ in (\ref{fund soln of heat eq}), the Neumann heat kernel $N(x,y,t)$ of $\O$ usually does not have an explicit formula. Nevertheless, when $t$ is small, $N(x,y,t)$ can be bounded in terms of $\Phi$.

\begin{lemma}\label{Lemma, quant of NHK}
There exists $C=C(n,\O)$ such that for any $x,y\in\ol{\O}$ and $t\in(0,1]$,
\be\label{quant of NHK}
0\leq N(x,y,t)\leq C\,\Phi(x-y,2t),\ee
where $\Phi$ is defined as in (\ref{fund soln of heat eq}).
\end{lemma}
\begin{proof}
Applying Theorem 3.2.9 on Page 90 of \cite{Dav89} with $\lambda=1$, $\mu=1$ and $\delta=1/2$ to the heat operator (\ref{heat op}). Then there exists $C_{1}=C_{1}(\O)$ such that
\[0\leq N(x,y,t)\leq C_{1}\max\{t^{-n/2},1\}\exp\bigg(-\frac{|x-y|^{2}}{6t}\bigg).\]
In particular, when $t\in(0,1]$, we have
\begin{align*}
0\leq N(x,y,t) &\leq C_{1}t^{-n/2}\exp\bigg(-\frac{|x-y|^{2}}{6t}\bigg) \\
&\leq 2^{n/2}C_{1}\Phi(x-y,2t).
\end{align*} \end{proof}

\subsection{Representation Formula By the Neumann Heat Kernel}
\label{Subsec, Rep formula by NHK}
One of the applications of the Neumann heat kernel is the representation formula of the solution to the heat equation with Neumann boundary conditions. As a heuristic argument, let's fix any $x\in\O$ and $t>0$ and pretend the solution $u$ to (\ref{Prob}) is sufficiently smooth. Then it follows from part (b) and (d) of Corollary \ref{Cor, prop of NHK} that 
\[(\p_{t}-\Delta_{y})N(x,y,t-\tau)=(\p_{t}-\Delta_{y})N(y,x,t-\tau)=0, \quad\forall\, y\in\ol{\O},\,0<\tau<t.\]
As a result, 
\[\int_{0}^{t}\int_{\O}(\p_{t}-\Delta_{y})N(x,y,t-\tau)\,u(y,\tau)\,dy\,d\tau=0.\]
Equivalently,
\[\int_{0}^{t}\int_{\O}(-\p_{\tau}-\Delta_{y})N(x,y,t-\tau)\,u(y,\tau)\,dy\,d\tau=0.\]
Now formally integrating by parts and taking advantage of (b), (c) and (d) in Corollary \ref{Cor, prop of NHK}, we obtain
\[\begin{split}
u(x,t)=& \int_{0}^{t}\int_{\O}N(x,y,t-\tau)\,(\p_{\tau}-\Delta_{y})u(y,\tau)\,dy\,d\tau+\int_{\O}N(x,y,t)\,u(y,0)\,dy\\
&+\int_{0}^{t}\int_{\p\O}N(x,y,t-\tau)\,\frac{\p u(y,\tau)}{\p n(y)}\,dS(y)\,d\tau. 
\end{split}\]
Keeping in mind that $u$ is the solution to (\ref{Prob}), so
\[u(x,t)=\int_{\O}N(x,y,t)u_{0}(y)\,dy+\int_{0}^{t}\int_{\Gamma_{1}}N(x,y,t-\tau)u^{q}(y,\tau)\,dS(y)\,d\tau.\]
This is the representation formula that is desired, but we still need to justify it rigorously. Since the proof is standard but tedious, we decide to put it into Appendix \ref{Sec, proof for rep formula}. Here we will just present the formal statement of the representation formula.

\begin{lemma}\label{Lemma, rep for soln, initial}
Let $u$ be the solution to (\ref{Prob}) and denote $T^{*}$ to be its lifespan. Then for any $(x,t)\in\ol{\O}\times(0,T^{*})$, 
\be\label{rep for soln, initial}
u(x,t)=\int_{\O}N(x,y,t)u_{0}(y)\,dy+\int_{0}^{t}\int_{\Gamma_{1}}N(x,y,t-\tau)u^{q}(y,\tau)\,dS(y)\,d\tau.\ee
\end{lemma}
\begin{proof}
See Appendix \ref{Sec, proof for rep formula}.
\end{proof}

\begin{corollary}\label{Cor, rep for soln}
Let $u$ be the solution to (\ref{Prob}) and denote $T^{*}$ to be its lifespan. Then for any $T\in [0,T^{*})$ and $(x,t)\in\ol{\O}\times(0,T^{*}-T)$,
\be\label{rep for soln}
u(x,T+t)=\int_{\O}N(x,y,t)u(y,T)\,dy+\int_{0}^{t}\int_{\Gamma_{1}}N(x,y,t-\tau)u^{q}(y,T+\tau)\,dS(y)\,d\tau.\ee
\end{corollary}
\begin{proof}
Regarding $u(\cdot,T)$ as the initial data $u_{0}(\cdot)$ and then applying Lemma \ref{Lemma, rep for soln, initial} leads to the conclusion. \end{proof}

\subsection{Boundary-Time Integral of the Heat Kernel of $\m{R}^{n}$}
\label{Subsec, bdry-time int est}
This section will provide the crucial estimates (\ref{bdry-time int est}) and (\ref{bdry-time int est, critical}) that will be used in the proofs of Lemma \ref{Lemma, growth rate, general} and Lemma \ref{Lemma, growth rate, critical}, while these two lemmas are the key ingredients in the proofs of Theorem \ref{Thm, lower bdd, general case} and Theorem \ref{Thm, lower bdd}. The three statements in this subsection have already essentially appeared in \cite{YZ16, YZ18}. But for completeness and preciseness, we still include their proofs here.

In the following, for any $\t{x}\in\m{R}^{n-1}$, we denote 
\be\label{ball in R^(n-1)}
B(\t{x},\rho)=\{\t{y}\in\m{R}^{n-1}: |\t{y}-\t{x}|<\rho\}\ee
to be the ball centered at $\t{x}$ in $\m{R}^{n-1}$ with radius $\rho$.

\begin{lemma}\label{Lemma, bdd for bdry-time int}
There exists $C=C(n,\O)$ such that for any $t>0$ and $x\in\ol{\O}$,
\[t^{1/2}\int_{\p\O}\Phi(x-y,t)\,dS(y)\leq C.\]
\end{lemma}
\begin{proof}
Since $\p\O$ in this paper is assumed to be $C^{2}$, there exist finitely many balls $\wt{B}_{i}\triangleq B(\t{z}_{i},r_{i})\subset \m{R}^{n-1}$ and $C^{2}$ mappings $\varphi_{i}:\wt{B}_{i}\rightarrow \m{R}\,(1\leq i\leq K)$ such that 
\[\p\O=\bigcup_{i=1}^{K}\Big\{\big(\t{y},\varphi_{i}(\t{y})\big):\t{y}\in\wt{B}_{i}\Big\}. \]
The total number $K$ only depends on $\O$. As a result, by writing $x=(\t{x},x_{n})$ and parametrizing $\p\O$, then
\begin{eqnarray*}
t^{1/2}\int_{\p\O}\Phi(x-y,t)\,dS(y) &=& C\,t^{-\frac{n-1}{2}}\int_{\p\O}\exp\bigg(-\frac{|x-y|^{2}}{4t}\bigg)\,dS(y) \\
&\leq & C\,t^{-\frac{n-1}{2}}\sum_{i=1}^{K}\int_{\wt{B}_{i}} \exp\bigg(-\frac{|\t{x}-\t{y}|^{2}+|x_{n}-\varphi_{i}(\t{y})|^{2}}{4t}\bigg)\,d\t{y}\\
&\leq & C\,\sum_{i=1}^{K}t^{-\frac{n-1}{2}}\int_{\wt{B}_{i}} \exp\bigg(-\frac{|\t{x}-\t{y}|^{2}}{4t}\bigg)\,d\t{y}\\
&\leq & C\,\sum_{i=1}^{K}t^{-\frac{n-1}{2}}\int_{\m{R}^{n-1}} \exp\bigg(-\frac{|\t{z}|^{2}}{4t}\bigg)\,d\t{z}\\
&=& CK,
\end{eqnarray*}
where the last equality is because $t^{-\frac{n-1}{2}}\int_{\m{R}^{n-1}} \exp\big(-\frac{|\t{z}|^{2}}{4t}\big)\,d\t{z}$ is a universal constant.
\end{proof}

\begin{lemma}\label{Lemma, bdry-time int, general}
For any $\alpha\in[0,\frac{1}{n-1})$, there exists $C=C(n,\O,\alpha)$ such that for any $x\in\ol{\O}$ and $t\geq 0$,
\be\label{bdry-time int est}
\int_{0}^{t}\int_{\Gamma_1}\Phi(x-y,\tau)\,dS(y)\,d\tau\leq
C\,|\Gamma_{1}|^{\alpha}\,t^{[1-(n-1)\alpha]/2}.\ee
\end{lemma}
\begin{proof}
By Holder's inequality, 
\begin{eqnarray}
\int_{\Gamma_{1}}\Phi(x-y,\tau)\,dS(y) &\leq & |\Gamma_{1}|^{\alpha}\,\bigg(\int_{\Gamma_{1}}\big[\Phi(x-y,\tau)\big]^{\frac{1}{1-\alpha}}\,dS(y)\bigg)^{1-\alpha} \notag\\
&=& C\,|\Gamma_{1}|^{\alpha}\,\bigg(\int_{\Gamma_{1}}\tau^{-\frac{n\alpha}{2(1-\alpha)}}\Phi\big(x-y,(1-\alpha)\tau\big)\,dS(y)\bigg)^{1-\alpha} \notag\\
&=& C\,|\Gamma_{1}|^{\alpha}\,\tau^{-n\alpha/2}\bigg(\int_{\Gamma_{1}}\Phi\big(x-y,(1-\alpha)\tau\big)\,dS(y)\bigg)^{1-\alpha} \label{Holder for bdry int}.
\end{eqnarray}
By applying Lemma \ref{Lemma, bdd for bdry-time int}, 
\begin{eqnarray}
\int_{\Gamma_{1}}\Phi\big(x-y,(1-\alpha)\tau\big)\,dS(y) &\leq & \int_{\p\O}\Phi\big(x-y,(1-\alpha)\tau\big)\,dS(y) \notag\\
&\leq & C\,\tau^{-1/2} \label{bdry int bdd}.
\end{eqnarray}
Plugging (\ref{bdry int bdd}) into (\ref{Holder for bdry int}), 
\begin{eqnarray*}
\int_{\Gamma_{1}}\Phi(x-y,\tau)\,dS(y)\leq C\,|\Gamma_{1}|^{\alpha}\,\tau^{-[1+(n-1)\alpha]/2}.
\end{eqnarray*}
Now integrating $\tau$ from 0 to $t$ yields (\ref{bdry-time int est}).
\end{proof}

In the above lemma, $\alpha$ is strictly less than $\frac{1}{n-1}$, so it is natural to ask what will happen when $\alpha=\frac{1}{n-1}$? If we simply plug $\alpha=\frac{1}{n-1}$ into (\ref{bdry-time int est}), then it leads to 
\be\label{bdry-time int, critical, guess}\int_{0}^{t}\int_{\Gamma_1}\Phi(x-y,\tau)\,dS(y)\,d\tau\leq
C\,|\Gamma_{1}|^{\frac{1}{n-1}}\ee
whose right hand side is an upper bound independent of $t$. However, (\ref{bdry-time int, critical, guess}) is not true in general. For example, in the simplest case when $\Gamma_{1}$ is flat and ball-shaped in a hypersurface and $x\in\Gamma_{1}$ (see Figure \ref{Fig, flat broken part}), then (\ref{bdry-time int, critical, guess}) fails when either $t\rightarrow\infty$ or $n=2$ and $|\Gamma_{1}|\rightarrow 0^{+}$.
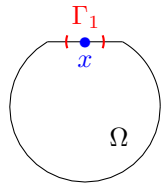
\begin{figure}[htbp]
\centering 
\begin{tikzpicture}[scale=1]
\draw [domain=0:60] plot ({cos(\x)}, {sin(\x)});
\draw [domain=120:360] plot ({cos(\x)}, {sin(\x)});
\draw (0.2,-0.4) node[right] {$\O$}; 
\draw [domain=-0.5:0.5] plot ({\x}, {0.866});
\draw [red, thick, domain=-20:20] plot ({0.25*cos(\x)}, {0.866+0.25*sin(\x)});
\draw [red, thick, domain=160:200] plot ({0.25*cos(\x)}, {0.866+0.25*sin(\x)});
\draw [red] (-0.3,1.2) node[right] {$\Gamma_{1}$}; 
\draw [blue](0,0.85) node {$\bullet$};
\draw [blue](0,0.8) node[below] {$x$};
\end{tikzpicture}
\caption{Flat $\Gamma_{1}$}
\label{Fig, flat broken part}
\end{figure}

\begin{itemize}
\item First, when $t\rightarrow \infty$, by parametrizing $\Gamma_{1}$ as what we will do in the proof for Proposition \ref{Prop, lower bdd for bdry-time int}, it follows from Lemma \ref{Lemma, radial bdry-time int, large dim} or Lemma \ref{Lemma, radial bdry-time int, dim 2} that the left hand side in (\ref{bdry-time int, critical, guess}) tends to infinity. 

\item Secondly, when $n=2$ and $|\Gamma_{1}|\rightarrow 0^{+}$, it follows similarly from Lemma \ref{Lemma, radial bdry-time int, dim 2} that 
\[\int_{0}^{t}\int_{\Gamma_1}\Phi(x-y,\tau)\,dS(y)\,d\tau\geq C|\Gamma_{1}|\bigg[1+\ln\bigg(\frac{t}{|\Gamma_{1}|^{2}}\bigg)\bigg],\]
which can not be bounded by $C|\Gamma_{1}|$.
\end{itemize}  
Consequently, we can expect (\ref{bdry-time int, critical, guess}) to hold only if $t$ is bounded and the right hand side of (\ref{bdry-time int, critical, guess}) adds an extra logarithmic term of $|\Gamma_{1}|^{-1}$ when $n=2$, see the following lemma.

\begin{lemma}\label{Lemma, bdry-time int, critical}
There exists $C=C(n,\O)$ such that for any $x\in\ol{\O}$, 
\be\label{bdry-time int est, critical}
\int_{0}^{2}\int_{\Gamma_1}\Phi(x-y,\tau)\,dS(y)\,d\tau\leq \left\{\begin{array}{ll}
C\,|\Gamma_{1}|^{\frac{1}{n-1}}, & \text{if}\quad n\geq 3, \vspace{0.1in}\\
C\,|\Gamma_{1}|\ln\Big(1+\dfrac{1}{|\Gamma_{1}|}\Big), & \text{if}\quad n=2.
\end{array}\right.\ee
\end{lemma}

\begin{proof}
When $n\geq 3$, we refer the readers to Lemma 2.7 in \cite{YZ18}. When $n=2$, it has been shown in Lemma 2.10 in \cite{YZ18} that 
\[\int_{0}^{1}\int_{\Gamma_1}\Phi(x-y,\tau)\,dS(y)\,d\tau\leq 
C\,|\Gamma_{1}|\ln\Big(1+\dfrac{1}{|\Gamma_{1}|}\Big).\]
So we only need to estimate $\int_{1}^{2}\int_{\Gamma_1}\Phi(x-y,\tau)\,dS(y)\,d\tau$. For $\tau\geq 1$, 
\[\Phi(x-y,\tau)\leq \frac{1}{(4\pi\tau)^{n/2}}\leq C.\]
Therefore, 
\[\int_{1}^{2}\int_{\Gamma_1}\Phi(x-y,\tau)\,dS(y)\,d\tau\leq C|\Gamma_{1}|\leq C|\Gamma_{1}|\,\ln\Big(1+\dfrac{1}{|\Gamma_{1}|}\Big),\]
where the last inequality is due to 
\[\ln\Big(1+\dfrac{1}{|\Gamma_{1}|}\Big)\geq \ln\Big(1+\dfrac{1}{|\p\O|}\Big).\]
\end{proof}
The order of the right hand side in (\ref{bdry-time int est, critical}) on $|\Gamma_{1}|$ is optimal as $|\Gamma_{1}|\rightarrow 0^{+}$, see Proposition \ref{Prop, lower bdd for bdry-time int} in Section \ref{Sec, sharp}.

\section{Proof of Theorem \ref{Thm, lower bdd, general case}}
\label{Sec, proof for general lower bdd}
In order to derive a lower bound for $T^{*}$, it is important to investigate how fast the solution $u$ can grow.

\begin{lemma}\label{Lemma, growth rate, general}
Let $u$ be the solution to (\ref{Prob}). Define $M(t)$ as in (\ref{max function at time t}). For any $\alpha\in[0,\frac{1}{n-1})$, there exists $C=C(n,\O,\alpha)$ such that for any $T\geq 0$ and $0\leq t<\min\{1,T^{*}-T\}$,  
\be\label{growth rate, general}
\frac{M(T+t)-M(T)}{M^{q}(T+t)}\leq C\,|\Gamma_{1}|^{\alpha}\,t^{[1-(n-1)\alpha]/2}.\ee
\end{lemma}
\begin{proof}
It is equivalent to prove 
\be\label{equiv for growth}
M(T+t)\leq M(T)+CM^{q}(T+t)\,|\Gamma_{1}|^{\alpha}\,t^{[1-(n-1)\alpha]/2}.\ee
For any $\sigma\in[0,T]$ and $x\in\ol{\O}$, it follows from the definition of $M(t)$ that 
\be\label{bdd for small time} u(x,\sigma)\leq M(T).\ee

In the rest of the proof, we assume $\sigma\in(T,T+t]$ and $x\in\ol{\O}$. By the representation formula (\ref{rep for soln}) with $t=\sigma-T$,
\begin{eqnarray*}
u(x,\sigma) &=& \int_{\O}N(x,y,\sigma-T)u(y,T)\,dy+\int_{0}^{\sigma-T}\int_{\Gamma_{1}}N(x,y,\sigma-T-\tau)u^{q}(y,T+\tau)\,dS(y)\,d\tau \\
&\leq & M(T)\int_{\O}N(x,y,\sigma-T)\,dy+M^{q}(\sigma)\int_{0}^{\sigma-T}\int_{\Gamma_{1}}N(x,y,\sigma-T-\tau)\,dS(y)\,d\tau.
\end{eqnarray*}
Applying part (e) in Corollary \ref{Cor, prop of NHK} and a change of variable in $\tau$, we get
\[u(x,\sigma)\leq M(T)+M^{q}(\sigma)\int_{0}^{\sigma-T}\int_{\Gamma_{1}}N(x,y,\tau)\,dS(y)\,d\tau.\]
Combining the nonnegativity of $N(x,y,\tau)$ in (\ref{quant of NHK}) and the fact that $\sigma\leq T+t$, we obtain  
\be\label{bdd by int of Green's fn} u(x,\sigma) \leq M(T)+M^{q}(T+t)\int_{0}^{t}\int_{\Gamma_{1}}N(x,y,\tau)\,dS(y)\,d\tau.\ee
Since $t\leq 1$, it follows from Lemma \ref{Lemma, quant of NHK} that
\begin{eqnarray*}
\int_{0}^{t}\int_{\Gamma_{1}}N(x,y,\tau)\,dS(y)\,d\tau &\leq & C \int_{0}^{t}\int_{\Gamma_{1}}\Phi(x-y,2\tau)\,dS(y)\,d\tau\\
&=& C \int_{0}^{2t}\int_{\Gamma_{1}}\Phi(x-y,\tau)\,dS(y)\,d\tau. \end{eqnarray*} 
Combining the above inequality with Lemma \ref{Lemma, bdry-time int, general}, we get
\be\label{key est}
\int_{0}^{t}\int_{\Gamma_{1}}N(x,y,\tau)\,dS(y)\,d\tau\leq C\,|\Gamma_{1}|^{\alpha}\,t^{[1-(n-1)\alpha]/2}.\ee
Plugging the above inequality into (\ref{bdd by int of Green's fn}), 
\be\label{bdd for large time}
u(x,\sigma)\leq M(T)+CM^{q}(T+t)\,|\Gamma_{1}|^{\alpha}\,t^{[1-(n-1)\alpha]/2}.\ee
Finally, (\ref{bdd for small time}) and (\ref{bdd for large time}) together lead to (\ref{equiv for growth}).
\end{proof}

In order to elaborate the proof of Theorem \ref{Thm, lower bdd, general case} more clearly, we give an elementary result as below.

\begin{lemma}\label{Lemma, lower bdd for a series}
Let $A>0$ and $0<\lambda<1$ be two constants. Then 
\be\label{lower bdd for a series}
\sum_{k=1}^{\infty}\min\{1,\lambda^{k}A\}\geq \frac{\ln(1+\lambda A)}{2\ln(\lambda^{-1})}.\ee
\end{lemma}
\begin{proof}
If $A\leq \frac{1}{\lambda}$, then 
\begin{eqnarray*}
\sum_{k=1}^{\infty}\min\{1,\lambda^{k}A\} = \sum_{k=1}^{\infty} \lambda^{k}A=\frac{\lambda A}{1-\lambda}\geq \frac{\ln(1+\lambda A)}{\ln(\lambda^{-1})},
\end{eqnarray*}
which implies (\ref{lower bdd for a series}).

If $A>\frac{1}{\lambda}$, then there exists $K\geq 1$ such that 
\be\label{split integer}
\lambda^{K}A>\frac{1}{2} \quad \text{and}\quad \lambda^{K+1}A\leq \frac{1}{2}.\ee
Thus, 
\be\label{large of the sum}
\sum_{k=1}^{\infty}\min\{1,\lambda^{k}A\} \geq \sum_{k=1}^{K}\min\{1,\lambda^{k}A\} \geq \frac{K}{2}.\ee
Since $\lambda^{K+1}A\leq \frac{1}{2}$, then 
\[K+1\geq \frac{\ln(2A)}{\ln(\lambda^{-1})}.\]
Noticing $\lambda A>1$, so
\be\label{large split integer}
K\geq \frac{\ln(2A)}{\ln(\lambda^{-1})}-1\geq \frac{\ln(1+\lambda A)}{\ln(\lambda^{-1})}.\ee
Plugging (\ref{large split integer}) into (\ref{large of the sum}) also yields (\ref{lower bdd for a series}).
\end{proof}

Now we start to prove Theorem \ref{Thm, lower bdd, general case}.

\begin{proof}[{\bf Proof of Theorem \ref{Thm, lower bdd, general case}}]
For $k\geq 0$, define 
\be\label{def of M_k} M_{k}=2^{k}M_{0}.\ee
Consider the function $M(t)$ defined in (\ref{max function at time t}). Denote $T_{k}$ to be the first time that $M(t)$ reaches $M_{k}$. That is 
\be\label{def of T_k}
T_{k}=\inf\{t\geq 0: M(t)\geq M_{k}\}.\ee
Since the solution $u$ to (\ref{Prob}) is continuous on $\ol{\O}\times[0,T^{*})$, 
\be \label{T_k is min}
T_{k}=\min\{t\geq 0: M(t)=M_{k}\}.\ee
In particular, $T_{0}=0$. 

Now for any $k\geq 1$, denote 
\[t_{k}=T_{k}-T_{k-1}.\] 
If $t_{k}<1$, then by Lemma \ref{Lemma, growth rate, general} (choose $T=T_{k-1}$, $t=t_{k}$ and $\alpha=\frac{1}{2(n-1)}$),
\be\label{growth rate, general proof}
\frac{M_{k}-M_{k-1}}{M_{k}^{q}}\leq C\,|\Gamma_{1}|^{\frac{1}{2(n-1)}}\,t_{k}^{1/4}.\ee
Plugging $M_{k}=2^{k}M_{0}$ and simplifying, we obtain
\be\label{lower bdd for t_k, general}
t_{k}\geq C\,(2^{k}M_{0})^{-4(q-1)}|\Gamma_{1}|^{-\frac{2}{n-1}}.\ee
Keeping in mind that (\ref{lower bdd for t_k, general}) is valid under the assumption that $t_{k}<1$. Thus, 
\begin{eqnarray*}
t_{k} &\geq & \min\Big\{1,\, C\,(2^{k}M_{0})^{-4(q-1)}|\Gamma_{1}|^{-\frac{2}{n-1}}\Big\}\\
&\geq & C\min\Big\{1,\, 2^{-4(q-1)k}M_{0}^{-4(q-1)}|\Gamma_{1}|^{-\frac{2}{n-1}}\Big\}.
\end{eqnarray*}

Applying Lemma \ref{Lemma, lower bdd for a series}  with 
\[\lambda =2^{-4(q-1)}\quad \text{and}\quad A=M_{0}^{-4(q-1)}|\Gamma_{1}|^{-\frac{2}{n-1}},\]
then
\begin{eqnarray*}
T^{*}=\sum_{k=1}^{\infty}t_{k}&\geq & C\,\sum_{k=1}^{\infty}\min\Big\{1,\, 2^{-4(q-1)k}M_{0}^{-4(q-1)}|\Gamma_{1}|^{-\frac{2}{n-1}}\Big\}\\
&\geq & \frac{C}{(q-1)}\,\ln\Big(1+(2M_{0})^{-4(q-1)}\,|\Gamma_{1}|^{-\frac{2}{n-1}}\Big).
\end{eqnarray*}
\end{proof}

\section{Proof of Theorem \ref{Thm, lower bdd}}\label{Sec, proof for main thm}
Let's first state an analogous result as Lemma \ref{Lemma, growth rate, general} concerning the growth rate of the solution. But this time it pushes to the critical power on $|\Gamma_{1}|$.

\begin{lemma}\label{Lemma, growth rate, critical}
Let $u$ be the solution to (\ref{Prob}). Define $M(t)$ as in (\ref{max function at time t}). Then there exists $C=C(n,\O)$ such that for any $T\geq 0$ and $0\leq t<\min\{1,T^{*}-T\}$,
\be\label{key ineq}
\frac{M(T+t)-M(T)}{M^{q}(T+t)}\leq 
\left\{\begin{array}{ll}
C\,|\Gamma_{1}|^{\frac{1}{n-1}}, & \text{if}\quad n\geq 3, \vspace{0.1in}\\
C\,|\Gamma_{1}|\ln\Big(1+\dfrac{1}{|\Gamma_{1}|}\Big), & \text{if}\quad n=2.
\end{array}\right.\ee
\end{lemma}
\begin{proof}
Noticing that $t<1$, so this proof is exactly the same as that of Lemma \ref{Lemma, growth rate, general}, except that Lemma \ref{Lemma, bdry-time int, critical} is needed instead of Lemma \ref{Lemma, bdry-time int, general} to estimate $\int_{0}^{t}\int_{\Gamma_{1}}N(x,y,\tau)\,dS(y)\,d\tau$ in (\ref{key est}) for $t<1$.
\end{proof}

In the rest of this section, we denote 
\be\label{def of E_q}
E_{q}=(q-1)^{q-1}/q^{q}, \quad\forall\, q>1.\ee
By elementary calculus, 
\be\label{est on E_q}
\frac{1}{3q}<E_{q}<\min\Big\{\frac{1}{q},\, \frac{1}{(q-1)\,e}\Big\}<1. \ee

\begin{lemma}\label{Lemma, criteria for step continue}
Fix any $q>1$ and $m>0$, denote $E_{q}$ as in (\ref{def of E_q}) and define $g:(m,\infty)\rightarrow \m{R}$ by
\be\label{auxiliary fcn}
g(\lambda)=\frac{\lambda-m}{\lambda^{q}},\quad\forall\,\lam>m.\ee
Then the following two claims hold.
\begin{itemize}
\item[(1)] For any $y\in\big(0, m^{1-q}E_{q}\big]$, there exists a unique $\lambda\in\big(m,\frac{q}{q-1}m\big]$ such that $g(\lambda)=y$.
\item[(2)] For any $y>m^{1-q}E_{q}$, there does not exist $\lambda>m$ such that $g(\lambda)=y$.
\end{itemize}
\end{lemma}
\begin{proof}
This is elementary, so we omit the proof. One can also see Lemma 4.1 in \cite{YZ18}.
\end{proof}

\begin{proof}[{\bf Proof of Theorem \ref{Thm, lower bdd}}]
We will demonstrate the detailed proof for the case $n\geq 3$ and briefly mention the case $n=2$ at the end since they are similar. In this proof, $C$ denote the constants which only depend on $n$ and $\O$, the values of $C$ may be different in different places. But $C^{*}$ will represent fixed constants which also only depend on $n$ and $\O$. $M(t)$ represents the same function as in (\ref{max function at time t}). The strategy of the proof is to find an appropriate finite increasing sequence $(M_{k})_{0\leq k\leq L}$ such that 
\be\label{large gap} T_{k}-T_{k-1}>1, \quad\forall\, 1\leq k\leq L,\ee
where $T_{k}$ is defined as
\[T_{k}=\min\{t\geq 0: M(t)=M_{k}\}.\]
After such a sequence is found, we will derive a lower bound for $L$ which is also a lower bound for $T^{*}$ due to (\ref{large gap}).

Let $n\geq 3$. Based on Lemma \ref{Lemma, growth rate, critical}, there exists a constant $C^{*}=C^{*}(n,\O)$ such that for any $T\geq 0$ and $0\leq t<\min\{1,T^{*}-T\}$, 
\be\label{key ineq, fixed const}
\frac{M(T+t)-M(T)}{M^{q}(T+t)}\leq 
C^{*}\,|\Gamma_{1}|^{\frac{1}{n-1}}.\ee
Define 
\be\label{def of delta_1}
\delta_{1}=2C^{*}|\Gamma_{1}|^{\frac{1}{n-1}}.\ee
Then we will use induction to construct a sequence $(M_{k})$ as below. 
\begin{itemize}
\item Define $M_{0}$ as in (\ref{initial max}).

\item Suppose $M_{k-1}$ has been constructed for some $k\geq 1$.
\begin{itemize}
\item[$\diamond$] If $M_{k-1}^{q-1}\delta_{1}\leq E_{q}$, then according to Lemma \ref{Lemma, criteria for step continue}, we define $M_{k}$ to be the unique solution such that
\be\label{smallness of M_k}
M_{k-1}<M_{k}\leq \frac{q}{q-1}M_{k-1} \ee
and
\be\label{def of M_k, critical}
\frac{M_{k}-M_{k-1}}{M_{k}^{q}}=\delta_{1}.\ee
 
\item[$\diamond$] If $M_{k-1}^{q-1}\delta_{1}> E_{q}$, then we do not define $M_{k}$ and stop the construction.
\end{itemize}
\end{itemize}
In the following, we will first show that the above construction stops after finite steps.

In fact, if the above construction continuous forever, then $M_{k-1}^{q-1}\,\delta_{1}\leq E_{q}$ for any $k\geq 1$. In addition, both (\ref{smallness of M_k}) and (\ref{def of M_k, critical}) hold. As a result, $(M_{k})_{k\geq 0}$ is a strictly increasing sequence and 
\beas M_{k}&=& M_{k-1}+M_{k}^{q}\,\delta_{1} \\
&\geq & M_{k-1}+ M_{k-1}M_{0}^{q-1}\,\delta_{1}\\
&=& (1+M_{0}^{q-1}\,\delta_{1})M_{k-1}.\eeas
Hence, 
\[M_{k}\geq (1+M_{0}^{q-1}\,\delta_{1})^{k}M_{0}\rightarrow \infty \quad\text{as}\quad k\rightarrow\infty,\]
which contradicts to the fact that $M_{k}^{q-1}\delta_{1}\leq E_{q}$. Thus, the inductive construction stops after finite steps and we denote the last term to be $M_{L}$.

For any $1\leq k\leq L$, write 
\[t_{k}=T_{k}-T_{k-1}.\]
If $t_{k}<1$, then plugging $T=T_{k-1}$ and $t=t_{k}$ into (\ref{key ineq, fixed const}) yields
\be\label{contradict of M_k}
\frac{M_{k}-M_{k-1}}{M_{k}^{q}}\leq C^{*}\,|\Gamma_{1}|^{\frac{1}{n-1}}.\ee
But this contradicts to the choice of $M_{k}$ in (\ref{def of M_k, critical}) due to the definition (\ref{def of delta_1}) for $\delta_{1}$. Hence, $t_{k}\geq 1$. As a result,
\be\label{T^* larger than L}
T^{*}\geq \sum_{k=1}^{L}t_{k}>L.\ee
The rest of the proof will provide a lower bound for $L$.

In fact,  we will prove that
\be\label{lower bdd for L}
L>\frac{1}{10(q-1)}\Big(\frac{1}{M_{0}^{q-1}\,\delta_{1}}-9q\Big).\ee
But before justifying this lower bound, let us first admit it and finish the proof of Theorem \ref{Thm, lower bdd}. Combining (\ref{T^* larger than L}) and (\ref{lower bdd for L}), 
\[T^{*}>\frac{1}{10(q-1)}\Big(\frac{1}{M_{0}^{q-1}\,\delta_{1}}-9q\Big).\]
Recalling the definition of $\delta_{1}$,
\[T^{*}>\frac{1}{10(q-1)}\Big(\frac{1}{2C^{*}\,M_{0}^{q-1}\,|\Gamma_{1}|^{1/(n-1)}}-9q\Big).\]
Denote $Y=M_{0}^{q-1}|\Gamma_{1}|^{\frac{1}{n-1}}$. Then 
\be\label{lower bdd for T^*, 1st}
T^{*}>\frac{1}{10(q-1)}\Big(\frac{1}{2C^{*}Y}-9q\Big).\ee
If 
\[Y\leq \frac{1}{36\,C^{*}q},\]
then $q\leq 1/(36C^{*}Y)$ and it follows from (\ref{lower bdd for T^*, 1st}) that
\[T^{*}\geq \frac{1}{40C^{*}(q-1)Y}.\]
Hence, we finish the proof of Theorem \ref{Thm, lower bdd}.

The remaining task is to verify (\ref{lower bdd for L}). If $M_{0}^{q-1}\delta_{1}>\frac{1}{9q}$, then (\ref{lower bdd for L}) holds automatically. So from now on, we assume $M_{0}^{q-1}\delta_{1}\leq \frac{1}{9q}$. Recalling (\ref{est on E_q}), we have
\be\label{small initial}
M_{0}^{q-1}\delta_{1}\leq \min\Big\{\frac{1}{2},\,E_{q}\Big\}.\ee
On the other hand, since the construction stops at $M_{L}$, then 
\be\label{large last term}
M_{L}^{q-1}\delta_{1}> E_{q}.\ee
Comparing (\ref{small initial}) and (\ref{large last term}), we conclude that $L$ is at least 1. As a result, there exists $1\leq L_{0}\leq L$ such that
\be\label{border index}
M_{L_{0}-1}^{q-1}\delta_{1}\leq \min\Big\{\frac{1}{2},\,E_{q}\Big\} \quad\text{and}\quad M_{L_{0}}^{q-1}\delta_{1}>\min\Big\{\frac{1}{2},\,E_{q}\Big\}.\ee 
The reason of considering $\min\big\{\frac{1}{2},\,E_{q}\big\}$ here instead of $E_{q}$ is because later we need the upper bound $\frac{1}{2}$ to justify (\ref{linearize}).
According to (\ref{def of M_k, critical}), 
\[M_{k-1}=M_{k}-M_{k}^{q} \delta_{1}=M_{k}\big(1-M_{k}^{q-1} \delta_{1}\big).\] 
Raising both sides to the power $q-1$ and multiplying by $\delta_{1}$,
\[M_{k-1}^{q-1}\delta_{1}=M_{k}^{q-1}\big(1-M_{k}^{q-1} \delta_{1}\big)^{q-1}\delta_{1}.\]
Define $x_{k}=M_{k}^{q-1}\delta_{1}$. Then $x_0=M_{0}^{q-1}\delta_{1}$ and
\be\label{iteration for x_k}
x_{k-1}=x_{k}\,(1-x_{k})^{q-1}, \quad\forall\, 1\leq k\leq L.\ee
Moreover, it follows from (\ref{border index}) that
\[x_{L_{0}-1}\leq \min\Big\{\frac{1}{2},\,E_{q}\Big\} \quad\text{and}\quad x_{L_{0}}>\min\Big\{\frac{1}{2},\,E_{q}\Big\}.\]

Now we claim the following inequality:
\begin{align}\label{ineq between last and first for x_k}
\frac{1}{x_{0}} \leq \frac{1}{x_{L_{0}-1}}+10(q-1)(L_{0}-1).
\end{align}
In fact, if $L_{0}=1$, then (\ref{ineq between last and first for x_k}) automatically holds. If $L_{0}\geq 2$, then for any $1\leq k\leq L_{0}-1$, we have
$0<x_{k}\leq x_{L_{0}-1}\leq 1/2$ and therefore,
\be\label{linearize}
x_{k-1}=x_{k}(1-x_{k})^{q-1}\geq x_{k}\big(1-2(q-1)x_{k}\big).\ee
Recalling the fact $x_{k}\leq x_{L_{0}-1}\leq E_{q}$ and the estimate $E_{q}<\frac{1}{(q-1)e}$ in (\ref{est on E_q}), then
\[1-2(q-1)x_{k}\geq 1-2(q-1)E_{q}\geq \frac{1}{5}.\]
Hence, taking the reciprocal in (\ref{linearize}) yields 
\begin{align}
\frac{1}{x_{k-1}} &\leq \frac{1}{x_{k}\big[1-2(q-1)x_{k}\big]} \notag\\
& =\frac{1}{x_{k}}+\frac{2(q-1)}{1-2(q-1)x_{k}} \notag\\
&\leq \frac{1}{x_{k}}+10(q-1). \label{iteration ineq}
\end{align}
Summing up (\ref{iteration ineq}) for $k$ from $1$ to $L_{0}-1$ yields (\ref{ineq between last and first for x_k}).  

Finally, since (\ref{smallness of M_k}) implies $M_{L_{0}}\leq \frac{q}{q-1}M_{L_{0}-1}$, then 
\[x_{L_{0}}=\bigg(\frac{M_{L_{0}}}{M_{L_{0}-1}}\bigg)^{q-1}x_{L_{0}-1}\leq \Big(\frac{q}{q-1}\Big)^{q-1}E_{q}=\frac{1}{q}.\]
Thus,
\[\frac{1}{3q}<\min\Big\{\frac{1}{2}, E_{q}\Big\}< x_{L_{0}}\leq \frac{1}{q},\]
Recalling (\ref{iteration for x_k}), then
\begin{align*}
x_{L_{0}-1} &= x_{L_{0}}(1-x_{L_{0}})^{q-1}\\
& >\frac{1}{3q}\Big(1-\frac{1}{q}\Big)^{q-1}=\frac{E_{q}}{3}.
\end{align*}
Plugging the above inequality and $x_{0}=M_{0}^{q-1}\delta_{1}$ into (\ref{ineq between last and first for x_k}), 
\begin{align*}
\frac{1}{M_{0}^{q-1}\delta_{1}} &< \frac{3}{E_{q}}+10(q-1)(L_{0}-1)\\
&< 9q+10(q-1)(L_{0}-1).
\end{align*}
Rearranging this inequality yields
\[L_{0}>\frac{1}{10(q-1)}\bigg(\frac{1}{M_{0}^{q-1}\delta_{1}}-9q\bigg)+1.\]
Hence, (\ref{lower bdd for L}) follows from $L\geq L_{0}$.

The proof for the case $n=2$ is almost the same except (due to Lemma \ref{Lemma, growth rate, critical}) changing $|\Gamma_{1}|^{\frac{1}{n-1}}$ to be $|\Gamma_{1}|\ln\Big(1+\dfrac{1}{|\Gamma_{1}|}\Big)$ in the above arguments.
\end{proof}

\section{Sharpness of the Key Estimate}
\label{Sec, sharp}
The key estimate Lemma \ref{Lemma, bdry-time int, critical} played an essential role in the proof of Lemma \ref{Lemma, growth rate, critical} (and therefore Theorem \ref{Thm, lower bdd}). Actually, the order on $|\Gamma_{1}|$ in Lemma \ref{Lemma, bdry-time int, critical} determines the order on $|\Gamma_{1}|$ of the lower bound of $T^{*}$ in Theorem \ref{Thm, lower bdd} as $|\Gamma_{1}|\rightarrow 0^{+}$. So it is desired to explore whether the order on $|\Gamma_{1}|$ in Lemma \ref{Lemma, bdry-time int, critical} is optimal as $|\Gamma_{1}|\rightarrow 0^{+}$? The goal of this section is to give an affirmative answer to this question when $\Gamma_{1}$, as a partial boundary of $\O$, is flat and ball-shaped (see e.g. Figure \ref{Fig, flat broken part}). As notation conventions, we denote 
\[\wt{B}(\rho)=\{(\t{x},0): \t{x}\in\m{R}^{n-1}, |\t{x}|<\rho\}\]
to be the flat ball centered at the origin on the hypersurface $\{x\in\m{R}^{n}: x_{n}=0\}$ with radius $\rho$. In addition, we use $\t{0}$ to represent the origin in $\m{R}^{n-1}$ and define $B(\t{x},\rho)$ as (\ref{ball in R^(n-1)}) to be the ball in $\m{R}^{n-1}$.

\begin{proposition}\label{Prop, lower bdd for bdry-time int}
There exists $C=C(n)$ such that for any $\Gamma_{1}=\wt{B}(\rho)$ with $0<\rho\leq 1$ and for any $x\in\ol{\Gamma}_{1}$,
\be\label{lower bdd for bdry-time int}
\int_{0}^{2}\int_{\Gamma_1}\Phi(x-y,t)\,dS(y)\,dt\geq \left\{\begin{array}{ll}
C\,|\Gamma_{1}|^{\frac{1}{n-1}}, & \text{if}\quad n\geq 3, \vspace{0.1in}\\
C\,|\Gamma_{1}|\ln\Big(1+\dfrac{1}{|\Gamma_{1}|}\Big), & \text{if}\quad n=2.
\end{array}\right.\ee
\end{proposition}

In order to prove Proposition \ref{Prop, lower bdd for bdry-time int}, we need two elementary results, Lemma \ref{Lemma, radial bdry-time int, large dim} and Lemma \ref{Lemma, radial bdry-time int, dim 2}, which may be of independent interest. In this section, for any positive integer $n\geq 2$, we define the function $\phi_{n}:\m{R}_{+}\times\m{R}_{+}\rightarrow \m{R}_{+}$ by 
\be\label{def of phi_n}
\phi_{n}(T,R)=\int_{0}^{T}\int_{0}^{R}\frac{r^{n-2}}{t^{n/2}}\,\exp\Big(-\frac{r^{2}}{4t}\Big)\,dr\,dt, \quad\forall\,T>0,\, R>0.\ee

\begin{lemma}\label{Lemma, radial bdry-time int, large dim}
Let $n\geq 3$ and define $\phi_{n}$ as (\ref{def of phi_n}). Then there exist $C_{1}=C_{1}(n)$ and $C_{2}=C_{2}(n)$ such that
\be\label{radial bdry-time int, large dim}
C_{1}\leq \frac{\phi_{n}(T,R)}{\min\{\sqrt{T},\,R\}}\leq C_{2}, \quad \forall\, T>0,\,R>0.\ee
\end{lemma}

\begin{lemma}\label{Lemma, radial bdry-time int, dim 2}
Define $\phi_{2}$ as (\ref{def of phi_n}) with $n=2$. Then there exist two universal constants $C_{1}$ and $C_{2}$ such that 
\be\label{radial bdry-time int, dim 2, small T}
C_{1}\leq \frac{\phi_{2}(T,R)}{\sqrt{T}}\leq C_{2}, \quad\forall\, 0<T<R^{2},\ee
and
\be\label{radial bdry-time int, dim 2, large T}
C_{1}\leq \frac{\phi_{2}(T,R)}{R\big[1+\ln\big(\frac{T}{R^{2}}\big)\big]}\leq C_{2}, \quad\forall\, 0<R^{2}\leq T.\ee
\end{lemma}

We will first prove Proposition \ref{Prop, lower bdd for bdry-time int} by admitting Lemma \ref{Lemma, radial bdry-time int, large dim} and Lemma \ref{Lemma, radial bdry-time int, dim 2}, and then justify these two lemmas at the end of this section.

\begin{proof}[{\bf Proof of Proposition \ref{Prop, lower bdd for bdry-time int}}]
Noticing that the surface area of $\Gamma_{1}$ is
\[|\Gamma_{1}|=C\rho^{n-1},\]
so it is equivalent to prove
\be\label{lower bdd for bdry-time int, radius}
\int_{0}^{2}\int_{\wt{B}(\rho)}\Phi(x-y,t)\,dS(y)\,dt\geq \left\{\begin{array}{ll}
C\rho, & \text{if}\quad n\geq 3, \vspace{0.1in}\\
C\rho\ln\Big(1+\dfrac{1}{\rho}\Big), & \text{if}\quad n=2.
\end{array}\right.\ee
Since $x\in\ol{\Gamma}_{1}$, we write $x=(\t{x},0)$, where $\t{x}\in\m{R}^{n-1}$ with $|\t{x}|\leq \rho$. Then by the parametrization $y=(\t{y},0)$ on $\wt{B}(\rho)$, we have
\begin{eqnarray*}
\int_{0}^{2}\int_{\wt{B}(\rho)}\Phi(x-y,t)\,dS(y)\,dt &=& C\int_{0}^{2}t^{-n/2}\int_{B(\t{0},\rho)}\exp\Big(-\frac{|\t{x}-\t{y}|^{2}}{4t}\Big)\,d\t{y}\,dt \\
&=& C\int_{0}^{2}t^{-n/2}\int_{B(\t{x},\rho)}\exp\Big(-\frac{|\t{y}|^{2}}{4t}\Big)\,d\t{y}\,dt.
\end{eqnarray*}
Since $|\t{x}|\leq \rho$, the overlap between $B(\t{x},\rho)$ and $B(\t{0},\rho)$ is comparable to $B(\t{0},\rho)$. For example, in Figure \ref{Fig, overlap} which shows the case when $n=3$, the overlap is at least one-third of $B(\t{0},\rho)$. 
\begin{figure}[htbp]
\centering 
\begin{tikzpicture}[scale=1]
\draw [thick, domain=0:360] plot ({cos(\x)}, {sin(\x)});
\draw [red, thick, domain=0:360] plot ({cos(\x)-1}, {sin(\x)});
\draw (0.8,1) node[right] {$B(\t{0},\rho)$}; 
\draw [red] (-1.8,1) node[left] {$B(\t{x},\rho)$};
\draw (0,0) node[right] {$\t{0}$}; 
\draw (0,0) node {$\bullet$}; 
\draw (-1,0) node[left] {$\t{x}$}; 
\draw [red] (-1,0) node {$\bullet$};
\end{tikzpicture}
\caption{Overlap}
\label{Fig, overlap}
\end{figure}
Also noticing that the function $\exp\big(-\frac{|\t{y}|^{2}}{4t}\big)$ is radial in $\t{y}$, so there exists a constant $C=C(n)$ such that
\[\int_{B(\t{x},\rho)}\exp\Big(-\frac{|\t{y}|^{2}}{4t}\Big)\,d\t{y}\geq C\int_{B(\t{0},\rho)}\exp\Big(-\frac{|\t{y}|^{2}}{4t}\Big)\,d\t{y}.\]
By polar coordinates, 
\[\int_{B(\t{0},\rho)}\exp\Big(-\frac{|\t{y}|^{2}}{4t}\Big)\,d\t{y}=C\int_{0}^{\rho}r^{n-2}\exp\Big(-\frac{r^{2}}{4t}\Big)\,dr.\]
Thus, 
\begin{eqnarray}
\int_{0}^{2}\int_{\wt{B}(\rho)}\Phi(x-y,t)\,dS(y)\,dt &\geq & C\int_{0}^{2}t^{-n/2}\int_{0}^{\rho}r^{n-2}\exp\Big(-\frac{r^{2}}{4t}\Big)\,dr \notag\\
&=& C\,\phi_{n}(2,\rho), \label{reduce to 1-d bdry-time int}
\end{eqnarray}
where $\phi_{n}$ is defined as in (\ref{def of phi_n}).

\begin{itemize}
\item If $n\geq 3$, then it follows from $0<\rho\leq 1$ and Lemma \ref{Lemma, radial bdry-time int, large dim} that 
\be\label{lower bdd for bdry-time int, radius, large dim}
\phi_{n}(2,\rho)\geq C\rho.\ee

\item If $n=2$, then it follows from $0<\rho\leq 1$ and Lemma \ref{Lemma, radial bdry-time int, dim 2} that
\begin{eqnarray}
\phi_{2}(2,\rho)&\geq & C\rho\Big[1+\ln\Big(\frac{2}{\rho^{2}}\Big)\Big] \notag\\
&\geq & C\rho\ln\Big(1+\frac{1}{\rho}\Big). \label{lower bdd for bdry-time int, radius, dim 2}
\end{eqnarray}
\end{itemize}
Combining (\ref{reduce to 1-d bdry-time int}), (\ref{lower bdd for bdry-time int, radius, large dim}) and (\ref{lower bdd for bdry-time int, radius, dim 2}) together yields (\ref{lower bdd for bdry-time int, radius}).
\end{proof}

Now we will prove Lemma \ref{Lemma, radial bdry-time int, large dim} and Lemma \ref{Lemma, radial bdry-time int, dim 2} which have been used in the above argument.

\begin{proof}[{\bf Proof of Lemma \ref{Lemma, radial bdry-time int, large dim}}]
We will verify the conclusion  by considering two cases $0<T<R^{2}$ and $0<R^{2}\leq T$.
\begin{itemize}
\item Case 1: Let $0<T<R^{2}$ (see Figure \ref{Fig, Case 1}).
\begin{align*}
\phi_{n}(T,R) &=\bigg(\iint_{I}+\iint_{II}\bigg)\frac{r^{n-2}}{t^{n/2}}\exp\Big(-\frac{r^{2}}{4t}\Big)\,dr\,dt \notag\\
&\triangleq g_{1}(T,R)+g_{2}(T,R).
\end{align*}

\begin{figure}[htbp]
\centering 
\begin{tikzpicture}[scale=1]
\draw [->] (0,0)--(4,0) node [anchor=north west]{$r$};
\draw [->] (0,0)--(0,4) node [anchor=south west] {$t$};
\path (2,0) coordinate (R);
\path (0,2.56) coordinate (T);
\path (2,2.56) coordinate (A);
\draw (R) node[below] {$R$};
\draw (T) node[left] {$T$};
\draw (R)--(A);
\draw (T)--(A);
\draw (0,0) node[below] {(0,0)};
\draw [red, thick, domain=0:1.6] plot({\x},{\x^2});
\draw (0,1) node[right] {\small $t=r^{2}$}; 
\draw (0.5,1.8) node[right] {$I$};
\draw (1.4,0.9) node[below] {$II$};
\end{tikzpicture}
\caption{Case 1}
\label{Fig, Case 1}
\end{figure}

Based on Figure \ref{Fig, Case 1}, 
\[g_{1}(T,R)=\int_{0}^{T}\int_{0}^{\sqrt{t}}\frac{r^{n-2}}{t^{n/2}}\exp\Big(-\frac{r^{2}}{4t}\Big)\,dr\,dt.\]
In this region, \[e^{-1/4}\leq \exp\Big(-\frac{r^{2}}{4t}\Big)\leq 1.\]
Also notice that
\[\int_{0}^{T}\int_{0}^{\sqrt{t}}\frac{r^{n-2}}{t^{n/2}}\,dr\,dt=\frac{2\sqrt{T}}{n-1}.\]
Therefore,
\be\label{est on 1, case 1}
\frac{2\sqrt{T}}{e^{1/4}(n-1)}\leq g_{1}(T,R)\leq \frac{2\sqrt{T}}{n-1}.\ee

Again according to Figure \ref{Fig, Case 1}, 
\[g_{2}(T,R) =\int_{0}^{T}\int_{\sqrt{t}}^{R}\frac{r^{n-2}}{t^{n/2}}\exp\Big(-\frac{r^{2}}{4t}\Big)\,dr\,dt. \]
Consider the inner integral and use the change of variable $\rho\triangleq \frac{r}{2\sqrt{t}}$ for $r$,
\begin{eqnarray*}\int_{\sqrt{t}}^{R}\frac{r^{n-2}}{t^{n/2}}\exp\Big(-\frac{r^{2}}{4t}\Big)\,dr &\leq &\int_{0}^{\infty}\frac{r^{n-2}}{t^{n/2}}\exp\Big(-\frac{r^{2}}{4t}\Big)\,dr \\
&=& 2^{n-1}t^{-1/2}\int_{0}^{\infty}\rho^{n-2}e^{-\rho^{2}}\,d\rho\\
&=& \frac{C}{t^{1/2}}.
\end{eqnarray*}
As a result, 
\be\label{est on 2, case 1}
0\leq g_{2}(T,R)\leq C\sqrt{T}.\ee
Combining (\ref{est on 1, case 1}) and (\ref{est on 2, case 1}), the estimate (\ref{radial bdry-time int, large dim}) is justified.

\item Case 2: Let $T\geq R^{2}$ (see Figure \ref{Fig, Case 2}).

\begin{eqnarray*}\label{int 2}
\phi_{n}(T,R) &=& \bigg(\iint_{I}+\iint_{II}\bigg)\frac{r^{n-2}}{t^{n/2}}\exp\Big(-\frac{r^{2}}{4t}\Big)\,dr\,dt \notag\\
&\triangleq & h_{1}(T,R)+h_{2}(T,R).
\end{eqnarray*}

\begin{figure}[htbp]
\centering 
\begin{tikzpicture}[scale=1]
\draw [->] (0,0)--(4,0) node [anchor=north west]{$r$};
\draw [->] (0,0)--(0,4) node [anchor=south west] {$t$};
\path (1.5,0) coordinate (R);
\path (0,3) coordinate (T);
\path (1.5,3) coordinate (A);
\draw (R) node[below] {$R$};
\draw (T) node[left] {$T$};
\draw (R)--(A);
\draw (T)--(A);
\draw (0,0) node[below] {(0,0)};
\draw [red, thick, domain=0:1.5] plot({\x},{\x^2});
\draw (0,1) node[right] {\small $t=r^{2}$}; 
\draw (0.5,2) node[right] {$I$};
\draw (1.2,0.6) node[below] {$II$};
\end{tikzpicture}
\caption{Case 2}
\label{Fig, Case 2}
\end{figure}

Based on Figure \ref{Fig, Case 2}, 
\[h_{2}(T,R)=\int_{0}^{R}\int_{0}^{r^{2}}\frac{r^{n-2}}{t^{n/2}}\exp\Big(-\frac{r^{2}}{4t}\Big)\,dt\,dr.\]
Consider the inner integral and use a change of variable $y\triangleq r^{2}/t$ for $t$,
\[\int_{0}^{r^2}\frac{r^{n-2}}{t^{n/2}}\exp\Big(-\frac{r^{2}}{4t}\Big)\,dt=\int_{1}^{\infty}y^{n/2-2}\,e^{-y/4}\,dy=C.\]
Therefore,
\be\label{est on 2, case 2} h_{2}(T,R)=CR.\ee

Again according to Figure \ref{Fig, Case 2},
\be\label{formula for h_1}
h_{1}(T,R) =\int_{0}^{R}\int_{r^{2}}^{T}\frac{r^{n-2}}{t^{n/2}}\exp\Big(-\frac{r^{2}}{4t}\Big)\,dt\,dr. \ee
In this integral region, 
\be\label{bound for exp}
e^{-1/4}\leq \exp\Big(-\frac{r^{2}}{4t}\Big)\leq 1.\ee
On the other hand, by direct calculation,
\be\label{calc, large n}\int_{0}^{R}\int_{r^{2}}^{T}\frac{r^{n-2}}{t^{n/2}}\,dt\,dr=\frac{2R}{n-2}\Big[1-\frac{1}{n-1}\Big(\frac{R^2}{T}\Big)^{n/2-1}\Big].\ee
Since $T\geq R^{2}$, then
\be\label{bound for int}
\frac{2R}{n-1}\leq \int_{0}^{R}\int_{r^{2}}^{T}\frac{r^{n-2}}{t^{n/2}}\,dt\,dr\leq \frac{2R}{n-2}.\ee
Combining (\ref{formula for h_1}), (\ref{bound for exp}) and (\ref{bound for int}), we obtain
\be\label{est on 1, case 2}
\frac{2R}{e^{1/4}(n-1)}\leq h_{1}(T,R)\leq \frac{2R}{n-2}.\ee

Hence, (\ref{est on 1, case 2}) and (\ref{est on 2, case 2}) together justifies (\ref{radial bdry-time int, large dim}).
\end{itemize}
\end{proof}

\begin{proof}[{\bf Proof of Lemma \ref{Lemma, radial bdry-time int, dim 2}}]
The proof for (\ref{radial bdry-time int, dim 2, small T}) follows the same argument as Case 1 in the above proof for Lemma \ref{Lemma, radial bdry-time int, large dim}. The proof for (\ref{radial bdry-time int, dim 2, large T}) also follows the same argument as Case 2 in the above proof except the estimate on the term $h_{1}(T,R)$ in (\ref{formula for h_1}). Let's rewrite $h_{1}(T,R)$ when $n=2$ as below.
\be\label{formula for h_1, dim 2}
h_{1}(T,R) =\int_{0}^{R}\int_{r^{2}}^{T}t^{-1}\exp\Big(-\frac{r^{2}}{4t}\Big)\,dt\,dr.\ee
By direct calculations, 
\be\label{calc, n=2}
\int_{0}^{R}\int_{r^{2}}^{T}t^{-1}\,dt\,dr=\int_{0}^{R}\ln(T)-2\ln(r)\,dr=R\Big[\ln\Big(\frac{T}{R^{2}}\Big)+2\Big].\ee
Combining (\ref{formula for h_1, dim 2}), (\ref{calc, n=2}) and (\ref{bound for exp}), we get 
\be\label{est on 1, case 2, border case}
e^{-1/4}R\Big[\ln\Big(\frac{T}{R^{2}}\Big)+2\Big]\leq h_{1}(T,R)\leq R\Big[\ln\Big(\frac{T}{R^{2}}\Big)+2\Big].\ee
Hence, (\ref{est on 1, case 2, border case}) and (\ref{est on 2, case 2}) together justifies (\ref{radial bdry-time int, dim 2, large T}).
\end{proof}

\begin{appendix}
\section*{Appendix}
\section{Proof of the Representation Formula}
\label{Sec, proof for rep formula}
The goal of this appendix is to justify the representation formula (\ref{rep for soln, initial}) in Lemma \ref{Lemma, rep for soln, initial}. First, notice that in part (c) of Corollary \ref{Cor, prop of NHK}, it claims that for any $\psi$ in (\ref{compatible initial}), 
\be\label{conv for compatible initial}
\lim_{t\rightarrow 0^{+}}\int_{\O}N(x,y,t)\psi(y)\,dy=\psi(x), \quad\text{uniformly in $x\in\ol{\O}$}.\ee
However, in many situations, the compatibility condition $\frac{\p \psi}{\p n}=0$ in (\ref{compatible initial}) may not be satisfied. So next, we will provide a convergence result for all the functions $\psi$ in $C(\ol{\O})$. But the convergence will only be pointwise and only valid for the inside point $x\in\O$.

\begin{lemma}\label{Lemma, delta fn against initial, general}
Let $N(x,y,t)$ be the Neumann heat kernel of $\O$ as in Definition \ref{Def, NHK}. Then for any $\psi\in C(\ol{\O})$ and for any $x\in\O$, 
\be
\lim_{t\rightarrow 0^{+}}\int_{\O}N(x,y,t)\psi(y)\,dy=\psi(x).\ee
\end{lemma}
\begin{proof}
Fix any $x\in\O$. Then the distance $d_{x}$ from $x$ to the boundary $\p\O$ is a fixed positive number. Let 
\[M=\max\limits_{y\in\ol{\O}}|\psi(y)|.\] 
Then $M$ is finite. Choose any mollifier $\eta(y)\in C^{\infty}(\m{R}^{n})$ with support in the unit ball and $\int_{\m{R}^{n}}\eta(y)\,dy=1$. For any $j\geq 1$, denote
\[\O_{1/j}=\Big\{y\in\O: \text{dist}(y,\,\p\O)>\frac{1}{j}\Big\}\]
and define $\varphi_{j}$ on $\m{R}^{n}$ as 
\[\varphi_{j}(y)=\left\{\begin{array}{ll}
\psi(y), & y\in\O_{1/j},\\
0, & y\notin \O_{1/j}.
\end{array}\right.\]
In addition, denote $\eta_{j}(y)=j^{n}\,\eta(jy)$ and define
$\psi_{j}(y)=(\eta_{2j}\ast\varphi_{j})(y)$. Then we know 
\begin{enumerate}[(a)]
\item $\psi_{j}\in C^{\infty}(\m{R}^{n})$ with support in $\O_{1/(2j)}$.
\item $|\psi_{j}(y)|\leq M$ for any $j\geq 1$ and $y\in\m{R}^{n}$.
\item For any compact subset $K$ in $\O$, $\psi_{j}$ uniformly converges to $\psi$ in $K$ as $j\rightarrow \infty$. 
\end{enumerate}

Now given any $\eps>0$, we choose a compact subset $K\subset\O$ such that $x\in K$, $\text{dist}(x,\p K)>d_{x}/2$ and the volume $|\O\setminus K|<\eps$ (see Figure \ref{Fig, Compact Subset $K$}).
\begin{figure}[htbp]
\centering 
\begin{tikzpicture}[scale=1]
\draw [thick, domain=0:180] plot ({3*cos(\x)}, {2*sin(\x)});
\draw [thick, domain=180:360] plot ({3*cos(\x)}, {2*sin(\x)});
\draw [red, thick, domain=0:180] plot ({2.6*cos(\x)}, {1.6*sin(\x)});
\draw [red, thick, domain=180:360] plot ({2.6*cos(\x)}, {1.6*sin(\x)});
\draw (-0.3,2.3) node[right] {$\O$}; 
\draw (-0.6,1.2) [red] node[right] {$K$}; 
\draw (-1.3,0.2) node[right] {$x$}; 
\draw (-1.3,0.3) node {$\bullet$};
\end{tikzpicture}
\caption{Compact Subset $K$}
\label{Fig, Compact Subset $K$}
\end{figure}
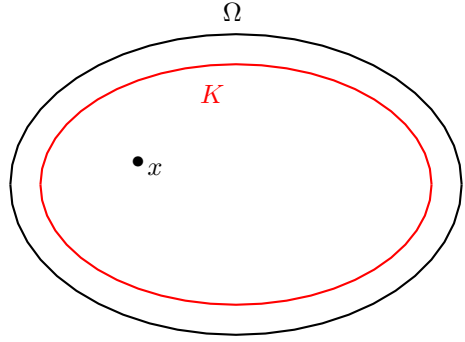
Fix this domain $K$. Then $\psi_{j}$ convergens to $\psi$ uniformly on $K$. Thus, there exists $J$ such that 
\be\label{uniform conv}
|\psi_{j}(y)-\psi(y)|<\eps,\quad\forall\, j\geq J,\, y\in K. \ee
Fix this $J$. Since $\psi_{J}$ satisfies (\ref{compatible initial}), then it follows from (\ref{conv for compatible initial}) that 
\[\lim_{t\rightarrow 0^{+}}\int_{\O}N(x,y,t)\psi_{J}(y)\,dy=\psi_{J}(x).\]
So there exists $0<\delta<1$ such that for any $t\in(0,\delta)$, 
\[\Big|\int_{\O}N(x,y,t)\psi_{J}(y)\,dy-\psi_{J}(x)\Big|<\eps.\]

Combining all the above results, we have that for any $t\in(0,\delta)$, 
\begin{eqnarray}
&&\quad \Big|\int_{\O}N(x,y,t)\psi(y)\,dy-\psi(x)\Big| \notag\\
&&\leq \int_{\O}N(x,y,t)\big|\psi(y)-\psi_{J}(y)\big|\,dy+\Big|\int_{\O}N(x,y,t)\psi_{J}(y)\,dy-\psi_{J}(x)\Big|+|\psi_{J}(x)-\psi(x)| \notag\\
&&\leq \int_{\O}N(x,y,t)\big|\psi(y)-\psi_{J}(y)\big|\,dy+2\eps. \label{three sum for conv}
\end{eqnarray}
In order to estimate the first  integral term, we split $\O$ into $K$ and $\O\setminus K$. On $K$, noticing (\ref{uniform conv}) and part (e) in Corollary \ref{Cor, prop of NHK}, so 
\be\label{est on K}
\int_{K}N(x,y,t)\big|\psi(y)-\psi_{J}(y)\big|\,dy \leq \eps\int_{K}N(x,y,t)\,dy\leq \eps. \ee
On $\O\setminus K$, since both $|\psi|$ and $|\psi_{J}|$ are bounded by $M$, 
\[\int_{\O\setminus K}N(x,y,t)\big|\psi(y)-\psi_{J}(y)\big|\,dy \leq 2M\int_{\O\setminus K}N(x,y,t)\,dy.\]
Then by Lemma \ref{Lemma, quant of NHK}, there exists a constant $C=C(n,\O)$ such that
\[\int_{\O\setminus K}N(x,y,t)\big|\psi(y)-\psi_{J}(y)\big|\,dy \leq 2CM\int_{\O\setminus K}\Phi(x-y,2t)\,dy.\]
Since $\text{dist}(x,\p K)>d_{x}/2$, then $|y-x|>d_{x}/2$ for any $y\in\O\setminus K$. Consequently,
\[\Phi(x-y,2t)\leq \frac{C}{t^{n/2}}\exp\Big(-\frac{d_{x}^{2}}{32t}\Big)\leq \frac{C}{d_{x}^{n}}.\]
Hence, 
\be\label{est on K complement}
\int_{\O\setminus K}N(x,y,t)\big|\psi(y)-\psi_{J}(y)\big|\,dy\leq \frac{CM}{d_{x}^{n}}\,|\O\setminus K|\leq \frac{CM\eps}{d_{x}^{n}}.\ee
Based on (\ref{three sum for conv}), (\ref{est on K}) and (\ref{est on K complement}), we conclude the proof.
\end{proof}

Before justifying the representation formula, it is helpful to discuss an auxiliary result. 

\begin{lemma}\label{Lemma, conv of bdry-time int of NHK}
Let $N(x,y,t)$ be the Neumann heat kernel of $\O$ as in Definition \ref{Def, NHK}. Let $\Gamma_{1}\subset\p\O$ be any part on the boundary. Then for any $x\in\ol{\O}$, $t>0$ and any continuous function $g$ on $\ol{\Gamma}_{1}\times[0,t]$, 
\[\lim_{\eps\rightarrow 0^{+}}\int_{0}^{t}\int_{\Gamma_1}N(x,y,t+\eps-\tau)\,g(y,\tau)\,dS(y)\,d\tau=\int_{0}^{t}\int_{\Gamma_1}N(x,y,t-\tau)\,g(y,\tau)\,dS(y)\,d\tau.\]
\end{lemma}
\begin{proof}
For any $0<\delta<\min\big\{\frac{t}{2}, \frac{1}{4}\big\}$, we split the integral $\int_{0}^{t}$ into $\int_{0}^{t-\delta}$ and $\int_{t-\delta}^{t}$. On $\int_{0}^{t-\delta}$, due to the continuity of $N$ and $g$, it is obvious that 
\[\lim_{\eps\rightarrow 0^{+}}\int_{0}^{t-\delta}\int_{\Gamma_1}N(x,y,t+\eps-\tau)\,g(y,\tau)\,dS(y)\,d\tau=\int_{0}^{t-\delta}\int_{\Gamma_1}N(x,y,t-\tau)\,g(y,\tau)\,dS(y)\,d\tau.\]
Hence, it suffices to prove 
\be\label{unif small}
\lim_{\delta\rightarrow 0^{+}}\int_{t-\delta}^{t}\int_{\Gamma_1}N(x,y,t+\eps-\tau)\,g(y,\tau)\,dS(y)\,d\tau=0 \quad\text{uniformly for $0\leq \eps\leq \frac{1}{4}$}.\ee
In fact, let 
\[M=\max_{\ol{\Gamma}_{1}\times[0,t]}|g|.\]
Then for $0\leq \eps\leq 1/4$ and $0<\delta<\min\big\{\frac{t}{2}, \frac{1}{4}\big\}$, it follows from Lemma \ref{Lemma, quant of NHK} that 
\[\bigg|\int_{t-\delta}^{t}\int_{\Gamma_1}N(x,y,t+\eps-\tau)\,g(y,\tau)\,dS(y)\,d\tau\bigg| \leq CM\int_{t-\delta}^{t}\int_{\Gamma_1}\Phi\big(x-y,2(t+\eps-\tau)\big)\,dS(y)\,d\tau\]
for some constant $C=C(n,\O)$. Applying Lemma \ref{Lemma, bdd for bdry-time int}, then 
\[\int_{\Gamma_1}\Phi\big(x-y,2(t+\eps-\tau)\big)\,dS(y)\leq C\,(t+\eps-\tau)^{-1/2}\leq C\,(t-\tau)^{-1/2}.\]
As a result,
\[\int_{t-\delta}^{t}\int_{\Gamma_1}\Phi\big(x-y,2(t+\eps-\tau)\big)\,dS(y)\,d\tau \leq C\int_{t-\delta}^{t}(t-\tau)^{-1/2}\,d\tau=2C\,\delta^{1/2},\]
which justifies (\ref{unif small}).
\end{proof}

\begin{proof}[{\bf Proof of Lemma \ref{Lemma, rep for soln, initial}}]
We will first consider the case $x\in\O$ and then the case $x\in\p\O$.
\begin{itemize}
\item Fix any $x\in\O$, $t\in(0,T^{*})$ and $\eps>0$. We define $\phi^{\eps}:\ol{\O}\times[0,t]\rightarrow\m{R}$ by
\[\phi^{\eps}(y,\tau)=N(x,y,t+\eps-\tau).\]
One can see that $\phi^{\eps}$ is $C^{2}$ in $y$ ($y\in\ol{\O}$) and $C^{1}$ in $\tau$ ($\tau\in [0,t]$). In addition, 
\begin{eqnarray*}
(\p_{\tau}+\Delta_{y})\phi^{\eps}(y,\tau)&=&(-\p_{t}+\Delta_{y})N(x,y,t+\eps-\tau)\\
&=& (-\p_{t}+\Delta_{y})N(y,x,t+\eps-\tau)=0.
\end{eqnarray*}
On the other hand, since $u$ is the classical solution to (\ref{Prob}), it is also the weak solution according to (\cite{YZ16}, Definition 3.4 and Theorem 3.7). Then by choosing $\phi=\phi^{\eps}$ in Definition 3.4 in \cite{YZ16}, we have 
\[\begin{split}
0=&\int_{\O}\phi^{\eps}(y,t)\,u(y,t)-\phi^{\eps}(y,0)\,u_0(y)\,dy -\int_{0}^{t}\int_{\Gamma_1}\phi^{\eps}(y,\tau)\,u^{q}(y,\tau)\,dS(y)\,d\tau\\
&+\int_{0}^{t}\int_{\p\O}u(y,\tau)\,\frac{\p\phi^{\eps}(y,\tau)}{\p n(y)}\,dS(y)\,d\tau.
\end{split} \]
Plugging $\phi^{\eps}(y,\tau)=N(x,y,t+\eps-\tau)$ into the above equality and noticing that 
\[\frac{\p\phi^{\eps}(y,\tau)}{\p n(y)}=\frac{\p N(x,y,t+\eps-\tau)}{\p n(y)}=\frac{\p N(y,x,t+\eps-\tau)}{\p n(y)}=0,\]
we obtain 
\[\begin{split}
0=&\int_{\O}N(x,y,\eps)\,u(y,t)-N(x,y,t+\eps)\,u_0(y)\,dy \\
&-\int_{0}^{t}\int_{\Gamma_1}N(x,y,t+\eps-\tau)\,u^{q}(y,\tau)\,dS(y)\,d\tau.
\end{split} \]
Sending $\eps\rightarrow 0^{+}$ and applying Lemma \ref{Lemma, delta fn against initial, general} and Lemma \ref{Lemma, conv of bdry-time int of NHK} yields (\ref{rep for soln, initial}).

\item Fix any $x\in\p\O$ and $t\in(0,T^{*})$. We choose a sequence $(x_{j})_{j\geq 1}$ such that $x_{j}\in\O$ and $x_{j}\rightarrow x$. Then from the above argument, it follows from (\ref{rep for soln, initial}) that for each $j\geq 1$,
\[u(x_{j},t)=\int_{\O}N(x_{j},y,t)u_{0}(y)\,dy+\int_{0}^{t}\int_{\Gamma_{1}}N(x_{j},y,t-\tau)u^{q}(y,\tau)\,dS(y)\,d\tau.\]
Since $t>0$, $N(x,y,t)$ is $C^{2}$ in $x$ on $\ol{\O}$. Sending $j\rightarrow\infty$, then $u(x_{j},t)$ converges to $u(x,t)$ and $\int_{\O}N(x_{j},y,t)u_{0}(y)\,dy$ converges to $\int_{\O}N(x,y,t)u_{0}(y)\,dy$. Finally, by similar argument as in the proof of Lemma \ref{Lemma, conv of bdry-time int of NHK}, we have
\[\int_{0}^{t}\int_{\Gamma_{1}}N(x_{j},y,t-\tau)u^{q}(y,\tau)\,dS(y)\,d\tau\rightarrow \int_{0}^{t}\int_{\Gamma_{1}}N(x,y,t-\tau)u^{q}(y,\tau)\,dS(y)\,d\tau.\]
\end{itemize}
\end{proof}
\end{appendix}

\bibliographystyle{plain}
\bibliography{Ref-Parabolic}

\end{document}